\documentclass[reqno,11pt]{amsart}
\usepackage{amsfonts}
\usepackage{a4wide}
\usepackage{color}
\usepackage{mathrsfs}
\usepackage{mathtools}
\usepackage{amsmath}
\usepackage{amssymb}
\usepackage{bbm}
\usepackage{esint}
\usepackage{nicefrac}
\usepackage{comment}
\numberwithin{equation}{section}
\usepackage[colorlinks,citecolor=green,linkcolor=red]{hyperref}

\usepackage[latin1]{inputenc}
\input xy
\xyoption{all}

\newtheorem{theorem}{Theorem}[section]
\newtheorem{Ca}[theorem]{Corollary}
\newtheorem{Th}[theorem]{Theorem}
\newtheorem{Lm}[theorem]{Lemma}
\newtheorem{Prop}[theorem]{Proposition}
\newtheorem{Def}[theorem]{Definition}
\newtheorem{Example}[theorem]{Example}
\newtheorem{Remark}[theorem]{Remark}
\newtheorem{Problem}[theorem]{Problem}

\title[Traces of Besov spaces: the limiting case]
{Traces of Besov spaces to regular subsets of metric measure spaces: the limiting case}
\author{Aleksei Y. Chikalov}
\address{Steklov Mathematical Institute of Russian Academy of Sciences}
\email{chikalov.a@phystech.edu}

\begin{document}
\allowdisplaybreaks


\begin{abstract}
Let $(X,d,\mu)$ be a metric measure space whose measure $\mu$ is uniformly locally doubling and which supports a local weak $(1,p)$-Poincar\'e inequality for some $p\in[1,\infty)$. Given $\theta\in(0,p)$ and an Ahlfors--David codimension-$\theta$ regular subset $E\subset X$, we provide a complete intrinsic description of the trace-space of the Besov space $B^{\theta/p}_{p,1}(X)$ to $E$. More precisely, we show that the trace operator is well defined and bounded from $B^{\theta/p}_{p,1}(X)$ to $L_p(E,\mathcal H_\theta\lfloor_E)$. We also show that the upper estimate in the Ahlfors--David codimension-$\theta$ regularity condition is necessary for such boundedness under the local weak Poincar\'e inequality. Conversely, assuming that $E$ is Ahlfors--David codimension-$\theta$ regular, we construct a bounded nonlinear extension operator from $L_p(E,\mathcal H_\theta\lfloor_E)$ to $B^{\theta/p}_{p,1}(X)$. Thus the trace-space is identified intrinsically with $L_p(E,\mathcal H_\theta\lfloor_E)$. This extends the classical limiting case of the trace theorem obtained by Burenkov and Gol'dman in \cite{bur_gold, Gol}. Finally, we apply the general theory to $K$-regular trees, $K\ge 1$, for which we additionally derive a necessary and sufficient criterion for the existence of traces.

\end{abstract}
\maketitle
\tableofcontents

\section{Introduction}

\noindent 

The theory of function spaces on metric measure spaces is an important and rapidly growing area of modern analysis.
It provides a flexible framework in which Sobolev, Newtonian, Besov and related spaces can be studied beyond the Euclidean setting.
We refer the interested reader to \cite{ AKZ, sob_mms, Jon} and references therein for general background and further results.
One of the central problems in this area is the trace problem, that is, the problem of finding sharp intrinsic descriptions of trace-spaces of function spaces to closed subsets of the ambient space.
In the present paper we study this problem for Besov spaces on metric measure spaces in the limiting case.
\par
In order to formulate the problem, we shall recall a few concepts. First of all, by a \emph{metric measure
space}, we mean a triple $(X, d, \mu)$, where $(X, d)$ is a \emph{complete separable metric space}, and $\mu$ is a \emph{locally finite Borel regular measure} with support $\operatorname{supp}\mu = X$. In order to develop a fruitful theory of function spaces on such a metric measure space, one needs to impose a compatibility condition between the distance and the measure. We assume that the metric measure space $X=(X, d, \mu)$ is \emph{$q$-admissible}, for some $q\in [1, \infty)$ (see Section~2 for the details), i.e.,
\begin{enumerate}
    \item the measure $\mu$ is uniformly locally doubling;
    \item the space $X$ supports a local weak $(1, q)$-Poincar\'e inequality.
\end{enumerate}
Given a closed set \(E\subset X\) and a locally finite Borel regular measure \(\nu\) on \(X\) with \(\operatorname{supp}\nu\subset E\), and given a locally integrable $f\in L_1^{\operatorname{loc}}(X)$, a \(\nu\)-measurable function \(\phi:E\to\mathbb R\) is called a trace of \(f\) to \(E\) if
\begin{equation}
\label{eq.intro_trace_def}
    \lim_{r\to0}
    \fint\limits_{B_r(x)}
    |f(y)-\phi(x)|\,d\mu(y)
    =
    0, \qquad \text{for \(\nu\)-a.e. \(x\in E\).}
\end{equation}
In this case, the corresponding equivalence class modulo \(\nu\)-negligible sets is denoted by \(\operatorname{Tr}f\).
\par
If $F(X)\subset L_1^{\operatorname{loc}}(X)$ is a normed linear space of functions, we say that the trace operator is well defined on $F(X)$ if, for each $f\in F(X)$,
there exists a trace of $f$ to $E$ in the sense of \eqref{eq.intro_trace_def}. In this case, the trace operator is defined as the mapping
\begin{equation}
    \operatorname{Tr}: F(X) \rightarrow L_0(E, \nu),
\end{equation}
where $L_0(E, \nu)$ denotes the space of all equivalence classes of $\nu$-measurable functions on $E$ modulo $\nu$-negligible sets. If the trace operator is well defined on $F(X)$, we can consider the trace-space
\begin{equation}
    F(X)\big|_E:=\left\{\operatorname{Tr}f: f\in F(X)\right\}.
\end{equation}
This space is equipped with the quotient-space norm, i.e.,
\begin{equation}
    \|\phi\|_{F(X)\big|_E}:= \inf\left\{ \|f\|_{F(X)}: f\in F(X), \operatorname{Tr}f = \phi\right\}.
\end{equation}
\par
The trace problem considered in this paper can be formulated as follows.
\begin{Problem}
\label{Pr.intro_trace_problem}
Let $p\in [1, \infty)$ and let $X = (X, d, \mu)$ be a $p$-admissible metric measure space. Let $E\subset X$ be a closed set, let \(\nu\) be a locally finite Borel regular measure on \(X\) with \(\operatorname{supp}\nu\subset E\), and let $F(X)\subset L_1^{\operatorname{loc}}(X)$ be a normed linear space such that the trace operator is well defined on $F(X)$.
\begin{enumerate}
    \item Given a function $\phi: E\to\mathbb{R}$, find a necessary and sufficient condition for the inclusion $\phi \in F(X)\big|_E$.
    \item Find an intrinsic norm on the trace space $F(X)\big|_E$ which is equivalent to the quotient-space norm.
    \item Decide whether every function in the trace-space $F(X)\big|_E$ admits a bounded extension to \(F(X)\).
\end{enumerate}
\end{Problem}
Without any regularity assumption on $(E, \nu)$ the trace problem is very challenging. In the present paper, we assume that $(E, \nu)$ is \emph{Ahlfors--David codimension-$\theta$ regular}, $\theta\in (0, \infty)$, i.e., there are constants $C_1, C_2>0$ such that
\begin{equation}
    C_1\frac{\mu(B_r(x))}{r^{\theta}} \le \nu (B_r(x))\le C_2\frac{\mu(B_r(x))}{r^{\theta}}, \qquad \text{for all } (x, r) \in E\times(0, 1],
\end{equation}
where $B_r(x)$ stands for the open ball centered at $x$ of radius $r>0$. In this case, $\nu$ is comparable to the codimension-$\theta$ Hausdorff measure
$\mathcal{H}_\theta\lfloor_E$ (see Section 2 for details).
We solve Problem~\ref{Pr.intro_trace_problem} for the Besov spaces $B^{{\theta}/{p}}_{p, 1}(X)$, $p\in [1, \infty)$, $\theta\in (0, p)$, assuming that $E$ is Ahlfors--David codimension-$\theta$ regular. In particular, we construct a \emph{bounded} nonlinear operator, called an \emph{extension} operator
\begin{equation}
    \operatorname{Ext}: B^{{\theta}/{p}}_{p, 1}(X)\big|_E \rightarrow B^{{\theta}/{p}}_{p, 1}(X),
\end{equation}
which is a right inverse to the trace operator.
\par
\noindent\textbf{Previously known results.}
Let us briefly recall the most relevant results concerning traces of Besov spaces.

\begin{enumerate}
    \item
    In the Euclidean setting, the classical result of Besov \cite{Bes_or} states that, given\\ \(m\in (0, n)\), \(p,q\in[1,\infty]\), and $s>\frac{n-m}{p}$, the trace-space of $B^s_{p,q}(\mathbb R^n)$ to $\mathbb{R}^m$ is linearly and continuously isomorphic to $ B^{s-\frac{n-m}p}_{p,q}(\mathbb R^m)$.
    \item
    The limiting case \(s=\frac{n-m}{p}\) is more delicate. In this case, for \(q>1\) the trace operator is not well defined on the corresponding Besov space.  For \(q=1\), Burenkov and Gol'dman (see \cite{bur_gold, Gol}) proved that, given \(0<m<n\), \(p\in[1,\infty)\) the trace-space $B^{\frac{n-m}{p}}_{p,1}(\mathbb R^n)\big|_{\mathbb R^m}$ coincides with $L_p(\mathbb R^m)$ as a linear space, and the corresponding norms are equivalent. One of the most remarkable features of this limiting case is that the extension operator has to be nonlinear (see \cite{bur_gold}).
    It should be mentioned that the limiting phenomenon was first discovered by Gagliardo in \cite{gagl}. Moreover, the same need for a nonlinear extension operator appears in that context (see \cite{koch_sn, Peetre}).
    \item
    Traces of classical Besov spaces to Ahlfors--David regular subsets of \(\mathbb R^n\) in the nonlimiting range were studied in several works (see, for example, \cite{Ihnat, Jon}).
    In this case, if \(E\subset\mathbb R^n\) is Ahlfors--David \(d\)-regular, $d\in (0, n)$, and \(s>\frac{n-d}{p}\), then the trace-space of $B^{s}_{p, q}(\mathbb{R}^n)$, $p, q\in [1, \infty)$, to $E$ is linearly and continuously isomorphic to the Besov space \(B^{s-\frac{n-d}{p}}_{p, q}(E)\).
    \item
    By contrast, the endpoint case \(s=\frac{n-d}{p}\), \(q=1\), is substantially more delicate. This case remains unresolved in general, even in the Euclidean case, although partial results are given in \cite{Gul}. In particular, given $p\in [1, \infty)$, $\theta \in (0, \infty)$, and an Ahlfors--David codimension-$\theta$ regular set $E \subset\mathbb{R}^n$, for each $f\in B^{\theta/p}_{p, 1}(\mathbb{R}^n)$, $\operatorname{Tr}f \in L_p(E, \mathcal{H}_{\theta}\lfloor_{E})$ and the trace operator
    \begin{equation}
        \operatorname{Tr}: B^{\theta/p}_{p, 1}(\mathbb{R}^n)\to L_p(E, \mathcal{H}_{\theta}\lfloor_{E}) 
    \end{equation}
    is bounded. Moreover, if $m\in \mathbb{N}$, $m \in (0, n)$, and $E$ is $m$-rectifiable, then there exists a bounded nonlinear extension operator
    \begin{equation}
        \operatorname{Ext}: L_p(E, \mathcal{H}_{n-m}\lfloor_E) \rightarrow B^{\frac{n-m}{p}}_{p, 1}(\mathbb{R}^n).
    \end{equation}
    Thus, $B^{\frac{n-m}{p}}_{p, 1}(\mathbb{R}^n)\big|_E = L_p(E, \mathcal{H}_{n-m}\lfloor_E)$ as linear spaces, corresponding norms being equivalent.
    \item
    Weighted versions of trace theorems for Besov spaces have also been studied. In particular,
traces of weighted Besov spaces with power weights and distance-type weights were considered
in \cite{Har_Sch, Iwon} in both limiting and nonlimiting cases. These results are of a
different nature from the present paper: they concern weighted Euclidean spaces, whereas our
main results are formulated for admissible metric measure spaces and codimension-$\theta$
regular subsets.
    \item
    In the setting of metric measure spaces, trace theorems for Besov spaces were obtained under
various regularity assumptions on the ambient space and on the set $E$. In the special case
when $(X,d,\mu)$ is an Ahlfors $Q$-regular space and $E\subset X$ is Ahlfors $d$-regular,
$d\in(0,Q)$, the trace problem for Besov spaces $B^s_{p,q}(X)$ with
$s>\frac{Q-d}{p}$, $p, q\in [1, \infty)$ was solved in \cite{mig,saks}. The paper \cite{saks} further relaxes these
assumptions by allowing admissible metric measure spaces and Ahlfors--David
codimension-$(Q-d)$ regular sets. However, these results concern the nonlimiting range,
whereas the endpoint case $s=\frac{\theta}{p}$, $q=1$, requires a different argument.
\end{enumerate}
    A few papers (\cite{Iwon, saks}) considered a wider range of parameters $p, q \in (0, \infty)$. However, in the present paper we restrict our attention to the case $p \in [1, \infty)$, $q = 1$. 
    \par
    The above discussion shows a clear analogy between the limiting trace phenomena for Sobolev spaces and for Besov spaces: in both cases the corresponding trace-spaces lose their smoothness and become $L_p$-spaces, and bounded extension operators are necessarily
nonlinear. At the same time, while trace theorems for Sobolev
spaces on metric measure spaces have been extensively developed (see, for example, \cite{mal_e, tyul2, tyul}), the corresponding endpoint problem for Besov spaces, especially beyond the
Euclidean setting, is much less understood. The purpose of the present paper is to provide a complete intrinsic description of the trace-space of $B^{\theta/p}_{p,1}(X)$ to $E\subset X$
for $p\in[1,\infty)$ and $\theta\in(0,p)$, assuming that $E$ is Ahlfors--David
codimension-$\theta$ regular.

\subsection{Main results}
Let $(X, d, \mu)$ be a metric measure space with a uniformly locally doubling measure $\mu$.
Until the end of the paper, $\overline{\Delta}_tf$, $t>0$ stands for the mean oscillation of a locally integrable function $f$, i.e.,
\begin{equation}
    \overline{\Delta}_tf(x) = \frac{1}{\mu(B_{t}(x))}\int\limits_{B_{t}(x)} \left|f(y) - \frac{1}{\mu(B_{t}(x))}\int\limits_{B_{t}(x)}f(z)d\mu(z)\right|d\mu(y).
\end{equation}
\par
Given $p, q \in [1, \infty]$ and $s \in (0, 1)$, $B^s_{p, q}(X)$ denotes the collection of all (equivalence classes of) functions $f \in L_p(X)$ such that
\begin{equation}
    \|f\|_{b^s_{p, q}(X)} := \left(\int\limits_{0}^{1} \left(t^{-s}\|\overline{\Delta}_tf\|_{L_p(X)}\right)^{q} \frac{dt}{t} \right)^{1/q} < \infty
\end{equation}
with the usual modification when $q = \infty$. We equip $B^s_{p, q}(X)$ with the norm
\begin{equation}
    \|f\|_{B^{s}_{p, q}(X)} := \|f\|_{L_p(X)} + \|f\|_{b^{s}_{p, q}(X)}.
\end{equation}
\par
\par
We now formulate the main results of the paper.
\begin{Th}[Direct trace theorem]
\label{Th.main_direct}
Let \((X,d,\mu)\) be a metric measure space with a uniformly locally doubling measure.
Let \(E\subset X\) be closed, and let \(\nu\) be a locally finite Borel regular measure on \(X\) with \(\operatorname{supp}\nu\subset E\).
Given $p\in [1, \infty)$ and $\theta\in (0, p)$, assume that \(\nu\) is upper codimension-\(\theta\) regular, that is, there exists a constant \(C>0\) such that
\begin{equation}
\label{eq.main_upper_regular}
    \nu(B_r(x)\cap E)
    \le
    C\frac{\mu(B_r(x))}{r^\theta},
    \qquad \text{for all } (x, r)\in E\times(0,1].
\end{equation}
Then every $f\in B^{\theta/p}_{p,1}(X)$ has a trace to $E$ in the sense of
\eqref{eq.intro_trace_def}, $\operatorname{Tr}f\in L_p(E,\nu)$, and the trace operator
\begin{equation}
\label{eq.main_direct_trace}
    \operatorname{Tr}:
    B^{\theta/p}_{p,1}(X)
    \rightarrow
    L_p(E,\nu)
\end{equation}
is bounded.
Moreover, for each \(f\in B^{\theta/p}_{p,1}(X)\),
\begin{equation}
\label{eq.main_direct_trace_limit}
    \lim_{r\to0}
    \fint\limits_{B_r(x)}
    |f(y)-\operatorname{Tr}f(x)|^p\,d\mu(y)
    =
    0
\end{equation}
for \(\nu\)-almost every \(x\in E\).
\end{Th}

The upper codimension-\(\theta\) regularity assumption is also necessary for the boundedness of the trace operator, provided the ambient space satisfies the local weak Poincar\'e inequality.

\begin{Th}[Necessity of upper regularity]
\label{Th.main_necessity}
Let \(X\) be $p$-admissible, $p\in [1, \infty)$, let \(E\subset X\) be closed, and let \(\nu\) be a locally finite Borel regular measure on \(X\) with \(\operatorname{supp}\nu \subset E\).
Given $\theta\in (0, p)$, assume that the trace operator
\begin{equation}
\label{eq.main_necessity_trace_operator}
    \operatorname{Tr}:
    B^{\theta/p}_{p,1}(X)
    \rightarrow
    L_p(E,\nu)
\end{equation}
is well defined and bounded.
Then \(\nu\) is upper codimension-\(\theta\) regular in the sense of \eqref{eq.main_upper_regular}.
\end{Th}

The inverse theorem requires the full Ahlfors--David codimension-\(\theta\) regularity of \(E\).

\begin{Th}[Extension theorem]
\label{Th.main_inverse}
Let \(X\) be $p$-admissible, $p\in [1, \infty)$, and let \(E\subset X\) be an Ahlfors--David codimension-\(\theta\) regular set, $\theta\in (0, p)$.
Then there exists a bounded nonlinear extension operator
\begin{equation}
\label{eq.main_extension_operator}
    \operatorname{Ext}:
    L_p(E,\mathcal H_\theta\lfloor_E)
    \rightarrow
    B^{\theta/p}_{p,1}(X)
\end{equation}
which is a right inverse to the trace operator, that is,
\begin{equation}
\label{eq.main_right_inverse}
    \operatorname{Tr}\circ\operatorname{Ext}
    =
    \operatorname{Id}_{L_p(E,\mathcal H_\theta\lfloor_E)}.
\end{equation}
\end{Th}

Combining Theorems~\ref{Th.main_direct} and~\ref{Th.main_inverse}, we obtain the following limiting case of the trace theorem.

\begin{Ca}
\label{Ca.main_trace_identification}
Let \(X\) be $p$-admissible, $p\in [1, \infty)$, and let \(E\subset X\) be an Ahlfors--David codimension-\(\theta\) regular set, $\theta\in (0, p)$.
Then
\begin{equation}
\label{eq.main_trace_identification}
    B^{\theta/p}_{p,1}(X)\big|_E
    =
    L_p(E,\mathcal H_\theta\lfloor_E)
\end{equation}
with equivalence of the quotient-space norm and the \(L_p(E,\mathcal H_\theta\lfloor_E)\)-norm.
\end{Ca}
\par
As a special case, we apply the general theory to weighted \(K\)-regular trees.
Let \(K\ge1\), and let \(X\) be the metric graph associated with a rooted
\(K\)-regular tree. We equip \(X\) with a radially symmetric measure and a path metric:
\begin{equation}
\label{eq.intro_tree_measure_metric}
    d\mu(x)=w(|x|)\,d\ell_G(x),
    \qquad
    d_\lambda(x,y)=\int_{[x,y]}\lambda(|z|)\,d\ell_G(z),
\end{equation}
where \(\ell_G\) is the one-dimensional length measure on the metric graph and
\(|x|\) denotes the graph distance from \(x\) to the root. We assume that $\mu(X)<\infty$ and $\operatorname{diam}_\lambda(X)<\infty$.
Then the completion \(\overline X\) is obtained from \(X\) by adding the boundary
\(\partial X\), which can be identified with the family of infinite geodesic rays
starting from the root. The boundary carries the natural probability measure \(\nu\), obtained by distributing unit mass
uniformly on $\partial X$.

There are two natural averaged trace operators on \(K\)-regular trees. The first one is
the metric trace, defined by averages over balls centered at boundary points in the
completed space \(\overline X\). The second one is the subtree trace, defined by averages
over subtrees approaching the boundary along a fixed geodesic ray. In Section~5 we prove
that, for the Besov functions considered below, these two trace operators coincide
\(\nu\)-almost everywhere. We denote their common value by \(\operatorname{Tr}f\).

An important feature of the trace problem on \(K\)-regular trees is that the boundary
measure \(\nu\) is fixed in advance. Thus, unlike in the general metric-measure setting
above, the regularity assumptions are formulated in terms of the ambient measure \(\mu\)
and the metric \(d_\lambda\).
For \(r\ge0\), put
\begin{equation}
\label{eq.intro_tree_tail_functions}
    M(r)
    :=
    \int_r^\infty w(t)K^{j(t)}\,dt,
    \qquad
    \rho(r)
    :=
    \int_r^\infty \lambda(t)\,dt,
\end{equation}
where \(j(t)=[t]+1\). We say that \((X,d_\lambda,\mu)\) is upper
\(\theta\)-regular if, for some constant $C>0$,
\begin{equation}
    M(r)\ge C \rho(r)^\theta,
    \qquad r\ge0,
\end{equation}
and lower \(\theta\)-regular if, for some constant $C>0$,
\begin{equation}
    M(r)\le C\rho(r)^\theta,
    \qquad r\ge0.
\end{equation}
These conditions are the radial counterparts of the upper and lower
codimension-\(\theta\) regularity assumptions from Section~2.

The tree version of the trace theorem is as follows.

\begin{Th}
\label{Th.trees_main}
Let \(p\in[1,\infty)\) and \(\theta\in(0,p)\). Let \(X\) be a weighted
\(K\)-regular tree as above, and assume that \(\mu\) is doubling on
\((X,d_\lambda)\).

If \((X,d_\lambda,\mu)\) is upper \(\theta\)-regular, then every
\(f\in B^{\theta/p}_{p,1}(X)\) has a trace to \(\partial X\),
\(\operatorname{Tr}f\in L_p(\partial X,\nu)\), and the trace operator
\begin{equation}
    \operatorname{Tr}:
    B^{\theta/p}_{p,1}(X)
    \rightarrow
    L_p(\partial X,\nu)
\end{equation}
is bounded.

Assume, in addition, that \((X,d_\lambda,\mu)\) supports a weak
\((1,p)\)-Poincar\'e inequality. Then, for every \(s\in(0,1)\), the following
conditions are equivalent:
\begin{enumerate}
    \item for every \(f\in B^s_{p,1}(X)\), the trace \(\operatorname{Tr}f(\xi)\) is
    finite for \(\nu\)-almost every \(\xi\in\partial X\);
    \item \((X,d_\lambda,\mu)\) is upper \(ps\)-regular.
\end{enumerate}

Finally, if \((X,d_\lambda,\mu)\) is both upper and lower \(\theta\)-regular and
supports a weak \((1,p)\)-Poincar\'e inequality, then there exists a bounded nonlinear
extension operator
\begin{equation}
    \operatorname{Ext}:
    L_p(\partial X,\nu)
    \rightarrow
    B^{\theta/p}_{p,1}(X)
\end{equation}
which is a right inverse to the trace operator.
\end{Th}
\subsection{Organization of the paper} This paper is organized as follows.
\par
\begin{itemize}
    \item In Section 2, we fix notation, give definitions, and state auxiliary results.
    \item In Section 3, we prove the direct trace theorem. In particular, we prove that the
    trace operator is well defined and bounded under the upper codimension-$\theta$ regularity
    assumption, and that this assumption is necessary for boundedness when $X$ is $p$-admissible.
    \item In Section 4, we prove the inverse trace theorem by constructing a bounded nonlinear
    extension operator and verifying the trace identity.
    \item In Section 5, we discuss traces of Besov spaces on $K$-regular trees.
\end{itemize}
\section{Preliminaries}
In this section we fix notation and collect several auxiliary facts used throughout the paper.
\par
Throughout the paper, $C$ denotes a positive inessential constant whose value may change from line to line. If the dependence on parameters is important, we write $C=C(a,b,c,\ldots)$. We write $A\lesssim B$ if $A\le CB$, and $A\approx B$ if both $A\lesssim B$ and $B\lesssim A$.

\subsection{Metric measure spaces and regular sets}

Let $(X,d)$ be a complete separable metric space. For $x\in X$ and $r>0$, we denote by $B_r(x)$ the open ball centered at $x$ of radius $r$, and by $\overline B_r(x)$ the corresponding closed ball. If $B=B_r(x)$ and $c>0$, we write $cB:=B_{cr}(x)$. When it is necessary to indicate the ambient space explicitly, we write $B_r^X(x)$.
For a set $G\subset X$, we denote $\operatorname{dist}(x,G):=\inf_{y\in G}d(x,y)$,
and put $B_r(G):=\{x\in X:\operatorname{dist}(x,G)<r\}$.
\par
By $\operatorname{LIP}(X)$ we denote the space of real-valued Lipschitz functions on $X$. If $f\in\operatorname{LIP}(X)$, then $L(f)$ denotes its Lipschitz constant. The local Lipschitz constant of $f$ is defined by
\begin{equation}
    \operatorname{lip}f(x)
    =
    \begin{cases}
        \displaystyle
        \limsup\limits_{y\to x}
        \frac{|f(y)-f(x)|}{d(y,x)},
        & \text{if $x$ is an accumulation point},\\[1.2ex]
        0,
        & \text{if $x$ is isolated}.
    \end{cases}
\end{equation}
\par
In what follows, a \emph{measure} $\mu$ on a complete separable metric space $X$ means a Borel regular nonzero locally finite outer measure on $X$. More precisely, $\operatorname{supp}\mu\neq\emptyset$ and $\mu(B_r(x))<\infty$ for all $(x, r) \in X\times(0, \infty)$.
Let $p\in[1,\infty]$, let $\mu$ be a measure on $X$, and let $G\subset X$ be measurable. We denote by $L_p(G,\mu)$ the space of all equivalence classes of measurable functions $f:G\to\mathbb R$ such that
\begin{equation}
    \|f\|_{L_p(G,\mu)}
    :=
    \left(
    \int\limits_G |f(x)|^p\,d\mu(x)
    \right)^{1/p}
    <\infty,
\end{equation}
with the standard modification when $p=\infty$.
As usual, $L_p^{\operatorname{loc}}(X,\mu)$ denotes the space of all equivalence classes of measurable functions $f:X\to\mathbb R$ such that $f\in L_p(B_r(x),\mu)$ for every $(x, r) \in X\times(0, \infty)$.
Given a measurable set $G\subset X$ with $0<\mu(G)<\infty$, we write, for each $f\in L_1(G,\mu)$,
\begin{equation}
    \fint\limits_G f(y)\,d\mu(y)
    :=
    \frac{1}{\mu(G)}
    \int\limits_G f(y)\,d\mu(y).
\end{equation}
Moreover, if $G=B_r(x)$, we write
\begin{equation}
    A_r^\mu f(x)
    :=
    \fint\limits_{B_r(x)} f(y)\,d\mu(y).
\end{equation}
\par
If \(E\subset X\) is closed and \(\nu\) is a Borel regular measure on \(X\) with
\(\operatorname{supp}\nu\subset E\), we shall use the notation \(L_p(E,\nu)\) and
\(L_0(E,\nu)\) for the spaces of \(\nu\)-equivalence classes of measurable functions
on \(E\). Equivalently, such functions may be viewed as functions on \(X\) defined
\(\nu\)-almost everywhere. Since \(\operatorname{supp}\nu\subset E\), we shall freely
write \(\nu(B_r(x))\) instead of \(\nu(B_r(x)\cap E)\).
\par
A \emph{metric measure space} is a triple $(X,d,\mu)$, where $(X,d)$ is a complete separable metric space and $\mu$ is a Borel regular measure giving positive finite mass to every ball. To develop a useful theory of function spaces on a metric measure space, one needs compatibility assumptions relating the distance and the measure.

\begin{Def}
Given $q\in[1,\infty)$, we say that a metric measure space $(X,d,\mu)$ is $q$-admissible and write $(X,d,\mu)\in\mathfrak U_q$ if the following two conditions hold.
\begin{enumerate}
    \item The measure $\mu$ is uniformly locally doubling, i.e., for every $R>0$ there exists $C_\mu(R)>0$ such that
    \begin{equation}
        \mu(B_{2r}(x))
        \le
        C_\mu(R)\mu(B_r(x)),
        \qquad (x, r)\in X\times(0,R].
    \end{equation}

    \item The space $(X,d,\mu)$ supports a local weak $(1,q)$-Poincar\'e inequality, i.e., for every $R>0$ there exist constants $C=C(R)>0$ and $\lambda=\lambda(R)\ge1$ such that, for every $f\in\operatorname{LIP}(X)$,
    \begin{equation}
        \fint\limits_{B_r(x)}
        |f(y)-A_r^\mu f(x)|\,d\mu(y)
        \le
        Cr
        \left(
        \fint\limits_{B_{\lambda r}(x)}
        (\operatorname{lip}f(y))^q\,d\mu(y)
        \right)^{1/q}
    \end{equation}
    for all $(x, r)\in X\times(0,R]$.
\end{enumerate}
When no confusion is possible, we shall simply write $X\in\mathfrak U_q$.
\end{Def}

Occasionally, we shall use the global counterparts of the above assumptions, where the same conditions hold with constants independent of $R$. In this case we write $X\in\mathfrak U_q^{\operatorname{gl}}$.
\par
The following elementary consequence of the local doubling property will be used repeatedly.

\begin{Lm}
\label{Lm.doubl_meas}
Let $(X,d,\mu)$ be a metric measure space with uniformly locally doubling measure. Then, for every $R>0$, there exist constants $c_1,c_2>0$, depending only on $R$, such that
\begin{equation}
    c_1\mu(B_r(x))
    \le
    \mu(B_r(y))
    \le
    c_2\mu(B_r(x))
\end{equation}
for all $r\in(0,R]$, all $x\in X$, and all $y\in B_r(x)$.
\end{Lm}

\begin{proof}
If $y\in B_r(x)$, then $B_r(x)\subset B_{2r}(y)$. Hence
\begin{equation}
    \mu(B_r(x))
    \le
    \mu(B_{2r}(y))
    \le
    C_\mu(R)\mu(B_r(y)).
\end{equation}
The reverse estimate follows by interchanging $x$ and $y$.
\end{proof}
\par
We shall also use the following finite-overlap consequence of local doubling.

\begin{Lm}[\cite{tyul}, Proposition 2.12]
\label{Lm.loc_d_m}
Let $(X,d,\mu)$ be a metric measure space with uniformly locally doubling measure. Then, for every $R>0$ and $c\ge1$, every ball $B_{cR}(x)$ contains at most
\(
    N_\mu(R,c)
    =
    C_\mu((c+1)R)^{\log_2(2c)+1}+1
\)
pairwise disjoint balls of radius $R$.
\end{Lm}
\par
By Lemma~\ref{Lm.loc_d_m}, every ball of finite radius is totally bounded. Since $X$ is complete, every closed ball \(\overline B_r(x)\), \(r<\infty\), is compact. Thus \(X\) is proper.
\par
Next, we recall the concept of the codimension-$\theta$ Hausdorff measure and Ahlfors--David codimension-$\theta$ regularity. Given $\theta\in[0,\infty)$, for each $E\subset X$ and each $\delta>0$, set
\begin{equation}
    \mathcal H_{\theta,\delta}(E)
    :=
    \inf
    \left\{
    \sum_i
    \frac{\mu(B_{r_i}(x_i))}{r_i^\theta}
    :
    E\subset\bigcup_i B_{r_i}(x_i),
    \quad
    0<r_i<\delta
    \right\},
\end{equation}
where the infimum is taken over all at most countable coverings of $E$ by balls $\{B_{r_i}(x_i)\}_i$. The codimension-$\theta$ Hausdorff measure of $E$ is defined by
\begin{equation}
    \mathcal H_\theta(E)
    :=
    \lim_{\delta\to0}
    \mathcal H_{\theta,\delta}(E).
\end{equation}
If \(E\subset X\) is Borel, we write \(\mathcal H_\theta\lfloor_E\) for the restriction of \(\mathcal H_\theta\) to \(E\), that is, \begin{equation} \label{eq.restricted_H_measure} (\mathcal H_\theta\lfloor_E)(G) := \mathcal H_\theta(G\cap E) \end{equation} for every Borel set \(G\subset X\). Thus \(\mathcal H_\theta\lfloor_E\) is a Borel measure on \(X\) with support contained in \(\bar{E}\).

\begin{Def}
Given $\theta\in[0,\infty)$, a closed set $E\subset X$ is called Ahlfors--David codimension-$\theta$ regular if there exist constants $c_1,c_2>0$ such that
\begin{equation}
\label{eq.coreg}
    c_1
    \frac{\mu(B_r(x))}{r^\theta}
    \le
    \mathcal H_\theta(B_r(x)\cap E)
    \le
    c_2
    \frac{\mu(B_r(x))}{r^\theta},
    \qquad x\in E,\quad r\in(0,1].
\end{equation}
\end{Def}
The following standard consequence of Ahlfors--David codimension-$\theta$ regularity will be useful in the sequel.
\begin{Lm}
\label{Lm.AD_zero_mu}
Let \(\theta>0\), and let \(E\subset X\) be Ahlfors--David codimension-\(\theta\) regular. Then \(\mu(E)=0\).
\end{Lm}
\begin{proof}
Fix \(x_0\in X\) and \(R>0\). Since \(X\) is proper, the set
\(E\cap \overline B_R(x_0)\) is compact. By Ahlfors--David regularity and compactness,
\(\mathcal H_\theta(E\cap \overline B_R(x_0))<\infty\). Let \(\delta>0\). By the definition of \(\mathcal H_{\theta,\delta}\), we can cover
\(E\cap \overline B_R(x_0)\) by balls \(\{B_{r_i}(x_i)\}_i\), \(0<r_i<\delta\), so that
\begin{equation}
    \sum_i \frac{\mu(B_{r_i}(x_i))}{r_i^\theta}
    \le
    \mathcal H_\theta(E\cap \overline B_R(x_0))+1.
\end{equation}
Therefore
\begin{equation}
    \mu(E\cap \overline B_R(x_0))
    \le
    \sum_i \mu(B_{r_i}(x_i))
    \le
    \delta^\theta
    \left(\mathcal H_\theta(E\cap \overline B_R(x_0))+1\right).
\end{equation}
Letting \(\delta\to0\), we obtain \(\mu(E\cap \overline B_R(x_0))=0\).
Since \(R>0\) was arbitrary, it follows that \(\mu(E)=0\).
\end{proof}

\begin{Remark}
If $\mu$ is uniformly locally doubling and $E\subset X$ is Ahlfors--David codimension-$\theta$ regular, $\theta\in [0, \infty)$, then $(E,d|_E,\mathcal{H}_\theta\lfloor_E)$ is a metric measure space. Moreover, $\mathcal{H}_\theta\lfloor_E$, regarded as a measure on the metric space $E$, is locally doubling.
\end{Remark}
\par
It will be useful to separate the two inequalities in \eqref{eq.coreg}.

\begin{Def}
Given $\theta\in[0,\infty)$, a closed set $E\subset X$, and a locally finite Borel regular measure $\nu$ on $X$ with $\operatorname{supp}\nu \subset E$, we say that $\nu$ is upper codimension-$\theta$ regular if there exists $C>0$ such that
\begin{equation}
\label{eq.up_reg}
    \nu(B_r(x))
    \le
    C
    \frac{\mu(B_r(x))}{r^\theta},
    \qquad \text{for all }(x, r)\in E\times(0,1].
\end{equation}
Similarly, $\nu$ is lower codimension-$\theta$ regular if there exists $C>0$ such that
\begin{equation}
\label{eq.low_reg}
    \frac{\mu(B_r(x))}{r^\theta}
    \le
    C\nu(B_r(x)),
    \qquad \text{for all }(x, r)\in E\times(0,1].
\end{equation}
\end{Def}

\begin{Remark}
\label{Rm.regularity}
Let \(E\subset X\) be closed and let \(\nu\) be a locally finite Borel regular measure
on \(X\) with \(\operatorname{supp}\nu\subset E\).
If \(\nu\) is upper codimension-\(\theta\) regular on \(E\), then there exists \(C>0\) such that
\begin{equation}
\label{eq.upper_implies_abs_cont}
    \nu(F)
    \le
    C\mathcal H_\theta(F)
\end{equation}
for every Borel set \(F\subset X\). In particular, $\nu\ll \mathcal H_\theta\lfloor_E$.
If \(\nu\) is lower codimension-\(\theta\) regular on \(E\), then there exists \(C>0\) such that
\begin{equation}
    \mathcal H_\theta(F)
    \le
    C\nu(F)
\end{equation}
for every Borel set \(F\subset X\).
\par
Consequently, if \(\nu\) is both upper and lower codimension-\(\theta\) regular on \(E\), then \(E\) is Ahlfors--David codimension-\(\theta\) regular and \(\nu\) is comparable with \(\mathcal H_\theta\lfloor_E\).
\end{Remark}
We shall need the following simple strengthening of the Lebesgue differentiation theorem (see \cite[Lemma 3.10]{gibara} and the discussion after the lemma).
\begin{Th}
\label{Th.leb_points}
Let $(X, d, \mu)$ be a metric measure space with uniformly locally doubling measure $\mu$. Let $f\in L_1^{\operatorname{loc}}(X,\mu)$ and $\theta\in(0,\infty)$. Define
\begin{equation}
    \Lambda_\theta
    :=
    \left\{
    x\in X:
    \limsup_{r\to0}
    r^\theta
    \fint\limits_{B_r(x)}
    |f(y)|\,d\mu(y)>0
    \right\}.
\end{equation}
Then $\mathcal H_\theta(\Lambda_\theta)=0$.
\end{Th}

\subsection{Besov spaces}

Given a metric measure space $(X,d,\mu)$, for each $f\in L_1^{\operatorname{loc}}(X,\mu)$ and each $t>0$, define the mean oscillation of $f$ at scale $t$ by
\begin{equation}
    \overline\Delta_t f(x)
    :=
    \fint\limits_{B_t(x)}
    |f(y)-A_t^\mu f(x)|\,d\mu(y).
\end{equation}

\begin{Def}
Given $s\in(0,1)$ and $p,q\in[1,\infty]$, the Besov space $B^s_{p,q}(X,\mu)$ consists of all $f\in L_p(X,\mu)$ such that
\begin{equation}
\label{eq.bes_sem}
    \|f\|_{b^s_{p,q}(X,\mu)}
    :=
    \left(
    \int\limits_0^1
    \left(
    t^{-s}
    \|\overline\Delta_t f\|_{L_p(X,\mu)}
    \right)^q
    \frac{dt}{t}
    \right)^{1/q}
    <
    \infty,
\end{equation}
with the usual modification when $q=\infty$. We equip this space with the norm
\begin{equation}
    \|f\|_{B^s_{p,q}(X,\mu)}
    :=
    \|f\|_{L_p(X,\mu)}
    +
    \|f\|_{b^s_{p,q}(X,\mu)}, \qquad f\in B^s_{p, q}(X, \mu).
\end{equation}
\end{Def}

\begin{Remark}
If the measure $\mu$ is locally doubling, one may replace the upper limit $1$ in \eqref{eq.bes_sem} by any finite scale. More precisely, for each $R>0$, replacing the upper limit $1$ by $R$ gives an equivalent full norm on $B^s_{p,q}(X,\mu)$.
\end{Remark}

\begin{Remark}
The case $s\ge1$ will not be considered. In particular, if $X\in\mathfrak U_p^{\operatorname{gl}}$ and $p\in(1,\infty)$, then $B^1_{p,q}(X,\mu)$ is trivial for all $q\in[1,\infty)$ (see \cite[Theorem 4.1]{AKZ}).
\end{Remark}

\begin{Remark}
\label{Rm.equiv_smooth}
If the measure $\mu$ is uniformly locally doubling, then the Besov norm can be described by several equivalent oscillation quantities. We shall use the following two. Given $\sigma\in(0,p]$, set, for each $f\in L_p^{\operatorname{loc}}(X,\mu)$,
\begin{equation}
    \overline\Delta_t^\sigma f(x)
    :=
    \left(
    \fint\limits_{B_t(x)}
    |f(y)-A_t^\mu f(x)|^\sigma\,d\mu(y)
    \right)^{1/\sigma}
\end{equation}
and
\begin{equation}
     \Delta_t^\sigma f(x)
    :=
    \left(
    \fint\limits_{B_t(x)}
    |f(y)-f(x)|^\sigma\,d\mu(y)
    \right)^{1/\sigma}.
\end{equation}
In the case \(\sigma=1\), we shall simply write $\Delta_t f:=\Delta_t^1 f$.
Replacing $\overline\Delta_t f$ in \eqref{eq.bes_sem} by $\overline\Delta_t^\sigma f$ or $\Delta_t^\sigma f$, with $\sigma\in(0,p]$, produces an equivalent norm. We refer to \cite{AKZ} for a more general result.
\end{Remark}

\begin{Remark}
\label{Rm.equiv_b_sem}
If the measure \(\mu\) is uniformly locally doubling, then
\begin{equation}
    \left(
    \sum_{k=0}^{\infty}
    \left(
    2^{ks}
    \|\overline\Delta_{2^{-k}}f\|_{L_p(X,\mu)}
    \right)^q
    \right)^{1/q}
    \approx
    \|f\|_{b^s_{p,q}(X,\mu)},
\end{equation}
with the usual modification when \(q=\infty\). The same equivalence holds with any of the oscillation quantities from Remark~\ref{Rm.equiv_smooth} in place of \(\overline\Delta_{2^{-k}}f\) (see, for example, \cite[Lemma 2.5]{AKZ}).
\end{Remark}

The following estimates are used to control the Besov seminorm of Lipschitz functions.

\begin{Lm}
\label{Lm.lip_est}
Let $(X,d,\mu)$ be a metric measure space with uniformly locally doubling measure. Let $p\in[1,\infty)$ and $R>0$. Then there exists $C=C(R)>0$ such that, for each $f\in L_p(X,\mu)$ and each $t\in(0,R]$,
\begin{equation}
\label{eq.osc_gen}
    \|\overline\Delta_t f\|_{L_p(X,\mu)}
    \le
    C\|f\|_{L_p(X,\mu)}.
\end{equation}
If, in addition, $X\in\mathfrak U_p$, then there exists $C=C(R)>0$ such that, for each $f\in\operatorname{LIP}(X)$ with $\operatorname{lip}f\in L_p(X,\mu)$ and each $t\in(0,R]$,
\begin{equation}
\label{eq.osc_lip}
    \|\overline\Delta_t f\|_{L_p(X,\mu)}
    \le
    Ct\|\operatorname{lip}f\|_{L_p(X,\mu)}.
\end{equation}
\end{Lm}

\begin{proof}
Since
\begin{equation}
    \overline\Delta_t f(x)
    \le
    2
    \fint\limits_{B_t(x)}
    |f(y)|\,d\mu(y),
\end{equation}
Jensen's inequality, Fubini's theorem, and Lemma~\ref{Lm.doubl_meas} give
\begin{equation}
    \int\limits_X
    \left(
    \fint\limits_{B_t(x)}
    |f(y)|\,d\mu(y)
    \right)^p
    d\mu(x)
    \le
    \int\limits_X
    |f(y)|^p
    \int\limits_{B_t(y)}
    \frac{d\mu(x)}{\mu(B_t(x))}
    d\mu(y)
    \lesssim
    \|f\|_{L_p(X,\mu)}^p.
\end{equation}
This proves \eqref{eq.osc_gen}.

Assume now that $X\in\mathfrak U_p$. Then, for each $t\in(0,R]$, the local weak $(1,p)$-Poincar\'e inequality gives
\begin{equation}
    \overline\Delta_t f(x)
    \lesssim
    t
    \left(
    \fint\limits_{B_{\lambda t}(x)}
    (\operatorname{lip}f(y))^p\,d\mu(y)
    \right)^{1/p}.
\end{equation}
Repeating the previous argument with $\operatorname{lip}f$ in place of $f$, we obtain
\begin{equation}
    \|\overline\Delta_t f\|_{L_p(X,\mu)}
    \lesssim
    t\|\operatorname{lip}f\|_{L_p(X,\mu)}.
\end{equation}
\end{proof}

\begin{Remark}
\label{Rm.lip_est}
Let $p,q\in[1,\infty)$ and $s\in(0,1)$. If $X\in\mathfrak U_p$, then Lemma~\ref{Lm.lip_est} implies that, for every $f\in\operatorname{LIP}(X)\cap L_p(X,\mu)$ with $\operatorname{lip}f\in L_p(X,\mu)$ and every $\delta\in(0,1)$,
\begin{equation}
\label{eq.bes_lip}
\begin{split}
    \|f\|_{b^s_{p,q}(X,\mu)}
    &\lesssim
    \left(
    \int\limits_0^\delta
    \|\operatorname{lip}f\|_{L_p(X,\mu)}^q
    \frac{dt}{t^{1+(s-1)q}}
    \right)^{1/q}
    +
    \left(
    \int\limits_\delta^1
    \|f\|_{L_p(X,\mu)}^q
    \frac{dt}{t^{1+sq}}
    \right)^{1/q} \\
    &\approx
    \delta^{1-s}
    \|\operatorname{lip}f\|_{L_p(X,\mu)}
    +
    \delta^{-s}
    \|f\|_{L_p(X,\mu)}.
\end{split}
\end{equation}
\end{Remark}
\par
\textbf{Important notation}. Whenever the metric measure space $X=(X,d,\mu)$ is fixed, we suppress the measure $\mu$ from the notation of function spaces. Thus, for every $p,q\in[1,\infty]$ and $s\in(0,1)$, we write
\begin{equation}
    L_p(X):=L_p(X,\mu),
    \qquad
    L_p^{\operatorname{loc}}(X):=L_p^{\operatorname{loc}}(X,\mu),
    \qquad
    B^s_{p,q}(X):=B^s_{p,q}(X,\mu).
\end{equation}
\section{The direct trace theorem}
Until the end of this section, we fix:
\begin{enumerate}
    \item a metric measure space $(X,d,\mu)$ with uniformly locally doubling measure $\mu$;
    \item an integrability parameter $p\in[1,\infty)$;
    \item a codimension parameter $\theta\in(0,p)$;
    \item a closed set $E\subset X$ and a locally finite Borel regular measure \(\nu\) on \(X\) with \(\operatorname{supp}\nu\subset E\).
\end{enumerate}
\par
First, we recall the notion of traces used in this paper.

\begin{Def}
\label{Def.tr}
Given $f\in L_1^{\operatorname{loc}}(X)$, a $\nu$-measurable function $\phi:E\to\mathbb R$ is called a trace of $f$ to $E$ if
\begin{equation}
\label{eq.tr_def}
    \lim_{r\to0}
    \fint\limits_{B_r(x)}
    |f(y)-\phi(x)|\,d\mu(y)
    =
    0,
    \qquad \text{for $\nu$-a.e. } x\in E.
\end{equation}
In this case, the equivalence class of $\phi$ modulo $\nu$-negligible sets is denoted by $\operatorname{Tr}f$.
\end{Def}

Given a normed function space $F(X)\subset L_1^{\operatorname{loc}}(X,\mu)$, we say that the \emph{trace operator} is well defined on $F(X)$ if, for each $f\in F(X)$, there exists a trace of $f$ to $E$ in the sense of Definition~\ref{Def.tr}. In this case, the mapping
\begin{equation}
    \operatorname{Tr}:F(X)\rightarrow L_0(E,\nu),
    \qquad
    f\mapsto \operatorname{Tr}f,
\end{equation}
is called the trace operator.

First, we show that the trace operator is well defined on \(B^{\theta/p}_{p,1}(X)\) under the upper codimension-\(\theta\) regularity assumption. Let us also mention that Lebesgue points and fine representatives of Besov
functions are often studied by capacitary methods (see, for example, \cite{Netrusov, Nuutinen}). However, the available capacitary estimates do not directly yield the Hausdorff-measure exceptional set estimate needed below in our metric-measure setting.

\begin{Th}
\label{Th.tr_ex}
Assume that $\nu$ is upper codimension-$\theta$ regular. Then the trace operator is well defined on $B^{\theta/p}_{p,1}(X)$. Moreover, for each $f\in B^{\theta/p}_{p,1}(X)$,
\begin{equation}
\label{eq.tr_ex_st}
    \lim_{r\to0}
    \fint\limits_{B_r(x)}
    |f(y)-\operatorname{Tr}f(x)|^p\,d\mu(y)
    =
    0,
    \qquad \text{for $\nu$-a.e. } x\in E.
\end{equation}
\end{Th}

We divide the proof into two auxiliary statements.

\begin{Prop}
\label{Prop.tr_ex_step1}
Under the assumption of Theorem~\ref{Th.tr_ex}, for each $f\in B^{\theta/p}_{p,1}(X)$, the limit
\begin{equation}
    \lim_{k\to\infty}
    \fint\limits_{B_{2^{-k}}(x)}
    f(y)\,d\mu(y)
\end{equation}
exists and is finite for $\nu$-almost every $x\in E$.
\end{Prop}

\begin{proof}
Fix $f\in B^{\theta/p}_{p,1}(X)$ and define
\begin{equation}
\label{eq.tr_d_st2}
    \Phi_f(x)
    :=
    \sum_{k=1}^{\infty}
    \left|
    A^\mu_{2^{-k-1}}f(x)
    -
    A^\mu_{2^{-k}}f(x)
    \right|,
    \qquad x\in E.
\end{equation}
We shall prove that $\Phi_f\in L_p(E,\nu)$. This will imply that the sequence $\left\{A^\mu_{2^{-k}}f(x)\right\}_{k=1}^{\infty}$
is Cauchy, and hence convergent, for $\nu$-almost every $x\in E$.
\par
For $x\in E$ and $k\in\mathbb N$, Jensen's inequality and Lemma~\ref{Lm.doubl_meas} give
\begin{equation}
\label{eq.tr_ex2}
\begin{split}
    \left|
    A^\mu_{2^{-k-1}}f(x)
    -
    A^\mu_{2^{-k}}f(x)
    \right|
    \le
    \fint\limits_{B_{2^{-k}}(x)}
    \fint\limits_{B_{2^{-k-1}}(x)}
    |f(z)-f(y)|\,d\mu(z)\,d\mu(y) \\
    \lesssim
    \fint\limits_{B_{2^{-k}}(x)}
    \fint\limits_{B_{2^{-k+1}}(y)}
    |f(z)-f(y)|\,d\mu(z)\,d\mu(y) =
    \fint\limits_{B_{2^{-k}}(x)}
    \Delta_{2^{-k+1}}f(y)\,d\mu(y).
\end{split}
\end{equation}
Consequently,
\begin{equation}
\label{eq.tr_ex3}
    \Phi_f(x)
    \lesssim
    \sum_{k=1}^{\infty}
    \fint\limits_{B_{2^{-k}}(x)}
    \Delta_{2^{-k+1}}f(y)\,d\mu(y),
    \qquad x\in E.
\end{equation}
\par
Let \(t\in(0,1/2]\) and \(h\in L_p(X)\). By Jensen's inequality and Fubini's theorem,
\begin{equation}
\label{eq.tr_ex4}
\begin{split}
    \int\limits_E
    \left|
    \fint\limits_{B_t(x)}
    h(y)\,d\mu(y)
    \right|^p
    d\nu(x)
    &\le
    \int\limits_E
    \fint\limits_{B_t(x)}
    |h(y)|^p\,d\mu(y)\,d\nu(x) \\
    &\le
    \int\limits_{B_t(E)}
    |h(y)|^p
    \int\limits_{E\cap B_t(y)}
    \frac{d\nu(x)}{\mu(B_t(x))}
    d\mu(y).
\end{split}
\end{equation}
If \(E\cap B_t(y)=\emptyset\), then the inner integral is zero. Otherwise, choose
\(\xi\in E\cap B_t(y)\). For every \(x\in E\cap B_t(y)\), Lemma~\ref{Lm.doubl_meas} gives $\mu(B_t(x))\approx \mu(B_t(\xi))$.
Moreover,
\begin{equation}
    E\cap B_t(y)\subset E\cap B_{2t}(\xi).
\end{equation}
Since \(2t\le1\), the upper codimension-\(\theta\) regularity of \(\nu\), together with the local doubling property of \(\mu\), gives
\begin{equation}
\label{eq.inner_av_est}
    \int\limits_{E\cap B_t(y)}
    \frac{d\nu(x)}{\mu(B_t(x))}
    \lesssim
    \frac{\nu(E\cap B_{2t}(\xi))}{\mu(B_t(\xi))}
    \lesssim
    t^{-\theta}.
\end{equation}
Combining \eqref{eq.tr_ex4} and \eqref{eq.inner_av_est}, we get
\begin{equation}
\label{eq.av_trace_est}
    \left\|
    A_t^\mu h
    \right\|_{L_p(E,\nu)}
    \lesssim
    t^{-\frac{\theta}{p}}
    \|h\|_{L_p(X)}.
\end{equation}
Applying this estimate with \(t=2^{-k}\) and \(h=\Delta_{2^{-k+1}}f\), and using \eqref{eq.tr_ex3}, we obtain
\begin{equation}
    \|\Phi_f\|_{L_p(E,\nu)}
    \lesssim
    \sum_{k=1}^{\infty}
    2^{k\theta/p}
    \|\Delta_{2^{-k+1}}f\|_{L_p(X)}
    \lesssim
    \|f\|_{b^{\theta/p}_{p,1}(X)}.
\end{equation}
Thus \(\Phi_f\in L_p(E,\nu)\). In particular, \(\Phi_f(x)<\infty\) for \(\nu\)-almost every \(x\in E\), and the desired dyadic limit exists for \(\nu\)-almost every \(x\in E\).
\end{proof}

Next, we show that the dyadic limit obtained above is indeed the trace of \(f\).

\begin{Lm}
\label{Lm.tr_ex_step2}
Assume that $\nu$ is upper codimension-$\theta$ regular. Let \(f\in B^{\theta/p}_{p,1}(X)\). Then the function
\begin{equation}
\label{eq.trace_dyadic_def}
    \phi(x)
    :=
    \lim_{k\to\infty}
    \fint\limits_{B_{2^{-k}}(x)}
    f(y)\,d\mu(y),
    \qquad x\in E,
\end{equation}
defined for \(\nu\)-almost every \(x\in E\), is a trace of \(f\) to \(E\). Moreover,
\begin{equation}
\label{eq.tr_ex_step2_st}
    \lim_{r\to0}
    \fint\limits_{B_r(x)}
    |f(y)-\phi(x)|^p\,d\mu(y)
    =
    0,
    \qquad \text{for $\nu$-a.e. } x\in E.
\end{equation}
\end{Lm}

\begin{proof}
We define \(\phi\) arbitrarily on the \(\nu\)-negligible set where the limit in \eqref{eq.trace_dyadic_def} does not exist. Then \(\phi\) is \(\nu\)-measurable as a pointwise almost everywhere limit of measurable functions.
\par
Let \(r\in(0,\frac12]\), and choose \(k=k(r)\in\mathbb N\) such that $2^{-k-1}<r\le 2^{-k}$.
Put \(t=2^{-k}\). Then \(r\le t<2r\). By the triangle inequality and the local doubling property of \(\mu\),
\begin{equation}
\label{eq.tr_d_ex5}
\begin{split}
    \fint\limits_{B_r(x)}
    |f(y)-\phi(x)|^p\,d\mu(y)
    &\lesssim
    \fint\limits_{B_t(x)}
    |f(y)-A_t^\mu f(x)|^p\,d\mu(y) \\
    &\quad+
    |A_t^\mu f(x)-\phi(x)|^p .
\end{split}
\end{equation}
The second term tends to \(0\) as \(r\to0\), equivalently \(t\to0\), by the definition of \(\phi\).
\par
It remains to estimate the first term. For each \(y\in B_t(x)\), we have \(B_t(x)\subset B_{2t}(y)\). Hence Jensen's inequality and Lemma~\ref{Lm.doubl_meas} imply
\begin{equation}
\label{eq.tr_d_ex6}
\begin{split}
    \fint\limits_{B_t(x)}
    |f(y)-A_t^\mu f(x)|^p\,d\mu(y)
    &\le
    \fint\limits_{B_t(x)}
    \fint\limits_{B_t(x)}
    |f(y)-f(z)|^p\,d\mu(z)\,d\mu(y) \\
    &\lesssim
    \fint\limits_{B_t(x)}
    \left(\Delta_{2t}^{p}f(y)\right)^p\,d\mu(y).
\end{split}
\end{equation}
Here \(\Delta^p_{2t}f\) denotes the oscillation quantity from Remark~\ref{Rm.equiv_smooth} with \(\sigma=p\).
\par
Set
\begin{equation}
    g(y)
    :=
    \sum_{m=0}^{\infty}
    2^{m\frac{\theta}{p}}
    \Delta_{2^{-m}}^{p}f(y),
    \qquad y\in X.
\end{equation}
By Remarks~\ref{Rm.equiv_smooth} and~\ref{Rm.equiv_b_sem}, we have \(g\in L_p(X)\). Since \(2t=2^{-k+1}\), it follows from the definition of \(g\) that
\begin{equation}
    \Delta_{2t}^{p}f(y)
    \lesssim t^{\frac{\theta}{p}}g(y).
\end{equation}
Consequently,
\begin{equation}
\label{eq.tr_d_eq7}
    \fint\limits_{B_t(x)}
    \left(\Delta_{2t}^{p}f(y)\right)^p\,d\mu(y)
    \lesssim
    t^\theta
    \fint\limits_{B_t(x)}
    g(y)^p\,d\mu(y).
\end{equation}
Since \(g^p\in L_1(X)\), Theorem~\ref{Th.leb_points} applied to \(g^p\) gives
\begin{equation}
    \lim_{t\to0}
    t^\theta
    \fint\limits_{B_t(x)}
    g(y)^p\,d\mu(y)
    =
    0
\end{equation}
for \(\mathcal H_\theta\)-almost every \(x\in X\).
\par
Finally, the upper codimension-\(\theta\) regularity of \(\nu\) implies that $\nu\ll \mathcal H_\theta\lfloor_E$ (see Remark~\ref{Rm.regularity}). Therefore the same convergence holds for \(\nu\)-almost every \(x\in E\). Combining this with \eqref{eq.tr_d_ex5}, \eqref{eq.tr_d_ex6}, and \eqref{eq.tr_d_eq7}, we obtain \eqref{eq.tr_ex_step2_st}.
\end{proof}

\begin{proof}[Proof of Theorem~\ref{Th.tr_ex}]
The theorem follows directly from Proposition~\ref{Prop.tr_ex_step1} and Lemma~\ref{Lm.tr_ex_step2}.
\end{proof}

Now we prove the direct part of the trace theorem.

\begin{Th}
\label{Th.dir}
Assume that $\nu$ is upper codimension-$\theta$ regular. Then, for each $f\in B^{\theta/p}_{p,1}(X)$, one has $\operatorname{Tr}f\in L_p(E,\nu)$. Moreover, the trace operator
\begin{equation}
\label{eq.tr_dir_def}
    \operatorname{Tr}:B^{\theta/p}_{p,1}(X)\rightarrow L_p(E,\nu)
\end{equation}
is bounded.
\par
If, in addition, \(X\in\mathfrak U_p\), then the following conditions are equivalent:
\begin{enumerate}
    \item the trace operator \(\operatorname{Tr}:B^{\theta/p}_{p,1}(X)\rightarrow L_p(E,\nu)\) is bounded;
    \item \(\nu\) is upper codimension-\(\theta\) regular.
\end{enumerate}
\end{Th}

\begin{proof}
A similar result for classical Besov spaces in the Euclidean setting was proved in \cite{Gul}.
\par
\emph{Step 1.} Assume first that \(\nu\) is upper codimension-\(\theta\) regular. By Theorem~\ref{Th.tr_ex}, the trace operator is well defined on \(B^{\theta/p}_{p,1}(X)\). Moreover, for every \(f\in B^{\theta/p}_{p,1}(X)\), Proposition~\ref{Prop.tr_ex_step1} gives, for \(\nu\)-almost every \(x\in E\),
\begin{equation}
\label{eq.trace_telescopic}
    \operatorname{Tr}f(x)
    =
    A^\mu_{2^{-1}}f(x)
    +
    \sum_{k=1}^{\infty}
    \left(
    A^\mu_{2^{-k-1}}f(x)
    -
    A^\mu_{2^{-k}}f(x)
    \right).
\end{equation}
Using \eqref{eq.av_trace_est} with \(t=2^{-1}\), and arguing as in the proof of Proposition~\ref{Prop.tr_ex_step1}, we obtain
\begin{equation}
\begin{split}
    \|\operatorname{Tr}f\|_{L_p(E,\nu)}
    \lesssim
    \|f\|_{L_p(X)}
    +
    \sum_{k=1}^{\infty}
    2^{k\frac{\theta}{p}}
    \|\Delta_{2^{-k+1}}f\|_{L_p(X)} \lesssim
    \|f\|_{B^{\theta/p}_{p,1}(X)}.
\end{split}
\end{equation}
Thus the trace operator is bounded.
\par
\emph{Step 2.} Assume now that \(X\in\mathfrak U_p\). By Step~1, the upper codimension-\(\theta\) regularity of \(\nu\) implies the boundedness of the trace operator. Conversely, suppose that the trace operator
\begin{equation}
    \operatorname{Tr}:B^{\theta/p}_{p,1}(X)\rightarrow L_p(E,\nu)
\end{equation}
is bounded. We prove the upper codimension-\(\theta\) regularity of \(\nu\). It is enough to prove the estimate for $r\in(0,1)$, the case $r=1$ follows by enlarging
the implicit constant.
\par
Take \(x_0\in E\) and \(r\in(0,1)\). Choose a Lipschitz cut-off function \(\psi\in\operatorname{LIP}(X)\) such that
\begin{equation}
\label{eq.lip_prop1}
    \chi_{B_r(x_0)}(x)
    \le
    \psi(x)
    \le
    \chi_{B_{2r}(x_0)}(x),
    \qquad x\in X,
\end{equation}
and
\begin{equation}
\label{eq.lip_prop2}
    L(\psi)\le \frac{C}{r},
\end{equation}
where \(C>0\) is independent of \(x_0\) and \(r\). For example, one may take
\begin{equation}
    \psi(x)
    =
    \max\left\{
    0,
    1-\frac{\operatorname{dist}(x,B_r(x_0))}{r}
    \right\}.
\end{equation}
Since \(\psi\) is Lipschitz, its trace is the pointwise restriction of \(\psi\) to \(E\). Indeed, for every \(x\in E\),
\begin{equation}
\fint\limits_{B_\rho(x)}
|\psi(y)-\psi(x)|\,d\mu(y)
\le
L(\psi)\rho
\to0
\qquad\text{as }\rho\to0.
\end{equation}
Hence
\begin{equation}
\label{eq.tr_d_nec1}
\|\operatorname{Tr}\psi\|_{L_p(E,\nu)}^p
\ge
\nu(B_r(x_0)\cap E).
\end{equation}
On the other hand, by \eqref{eq.lip_prop1}, \eqref{eq.lip_prop2}, and the local doubling property of \(\mu\),
\begin{equation}
\label{eq.cutoff_Lp_lip}
\|\psi\|_{L_p(X)}
\lesssim
\left(\mu(B_r(x_0))\right)^{1/p},
\qquad
\|\operatorname{lip}\psi\|_{L_p(X)}
\lesssim
r^{-1}\left(\mu(B_r(x_0))\right)^{1/p}.
\end{equation}
Since \(X\in\mathfrak U_p\), Remark~\ref{Rm.lip_est}, applied with \(s=\frac{\theta}{p}\) and $\delta = r$, gives
\begin{equation}
\label{eq.tr_d_nec2}
\|\psi\|_{B^{\theta/p}_{p,1}(X)}^p
\lesssim
\frac{\mu(B_r(x_0))}{r^\theta}.
\end{equation}
Therefore, by the boundedness of the trace operator,
\begin{equation}
\nu(B_r(x_0)\cap E)
\le
\|\operatorname{Tr}\psi\|_{L_p(E,\nu)}^p
\lesssim
\|\psi\|_{B^{\theta/p}_{p,1}(X)}^p
\lesssim
\frac{\mu(B_r(x_0))}{r^\theta}.
\end{equation}
This proves the upper codimension-\(\theta\) regularity of \(\nu\), and the theorem follows.
\end{proof}

\begin{Remark}
Theorem~\ref{Th.dir} shows that, under the additional assumption \(X\in\mathfrak U_p\), the upper codimension-\(\theta\) regularity of \(\nu\) is necessary for the boundedness of the trace operator
\begin{equation}
\operatorname{Tr}: B^{\theta/p}_{p,1}(X)\rightarrow L_p(E,\nu).
\end{equation}
However, this regularity is not necessary for the mere existence of traces.
\par
For instance, suppose that
\begin{equation}
E=\bigcup_{j=1}^{\infty}E_j,
\end{equation}
where the sets \(E_j\subset E\) are closed and \(\nu\)-measurable, and assume that for every \(j\in\mathbb N\) there exists a constant \(C_j>0\) such that
\begin{equation}
\label{eq.rem_d_tr_t}
\nu(B_r(x)\cap E_j)
\le
C_j
\frac{\mu(B_r(x))}{r^\theta},
\qquad \text{for all } (x, r)\in E_j\times(0,1].
\end{equation}
If the optimal constants in \eqref{eq.rem_d_tr_t} are unbounded, then \(\nu\) need not be upper codimension-\(\theta\) regular on the whole set \(E\). Nevertheless, by Theorem~\ref{Th.tr_ex} applied to each restriction \(\nu\lfloor_{E_j}\), every \(f\in B^{\theta/p}_{p,1}(X)\) has a trace to \(E_j\) for every \(j\in\mathbb N\), and hence has a trace to \(E\) for \(\nu\)-almost every point of \(E\).
\end{Remark}
\section{The inverse trace theorem}
Throughout this section, we fix the following data
\begin{enumerate}
\item  parameters $p\in[1,\infty)$ and $\theta\in(0,p)$;
\item \(X=(X,d,\mu)\in\mathfrak U_p\);
\item \(E\subset X\) is an Ahlfors--David codimension-\(\theta\) regular set.
\end{enumerate}

Our goal is to construct a bounded nonlinear extension operator
\begin{equation}
\label{eq.ext_goal}
\operatorname{Ext}:
L_p(E,\mathcal H_\theta\lfloor_E)
\rightarrow
B^{\theta/p}_{p,1}(X),
\end{equation}
which is a right inverse to the trace operator. This nonlinearity is not merely a feature of the construction: already in the classical endpoint case, one cannot expect a bounded linear extension operator (see \cite{bur_gold}).
\par
Unlike the direct trace theorem, the inverse theorem uses both sides of the Ahlfors--David codimension-\(\theta\) regularity. The lower codimension-\(\theta\) regularity is used to control the \(B^{\theta/p}_{p,1}(X)\)-norm of the extension, whereas the upper codimension-\(\theta\) regularity is used in the verification of the trace identity. The construction follows the Whitney-type approach from \cite[Sections~6--7]{mal_e}.
\par
We begin with a Whitney-type decomposition of a neighborhood of a closed set.

\begin{Th}
\label{Th.whit}
Let \(S\subset X\) be closed. Then there exists a family of balls
\(
\{B_{j,i}=B_{r_{j,i}}(p_{j,i})\}_{\substack{j\ge -1\\ i\in\mathcal I_j}},
\)
where each \(\mathcal I_j\) is at most countable, such that
\begin{enumerate}
\item
\begin{equation}
\{x\in X:\operatorname{dist}(x,S)\in(0,1)\}
\subset
\bigcup_{j\ge -1}\bigcup_{i\in\mathcal I_j}B_{j,i};
\end{equation}
\item for every \(\lambda\ge1\), there exists \(C(\lambda)>0\) such that
\begin{equation}
    \sum_{j\ge -1}\sum_{i\in\mathcal I_j}
    \chi_{\lambda B_{j,i}}
    \le
    C(\lambda);
\end{equation}
\item
\begin{equation}
    2^{-j-1}<r_{j,i}\le2^{-j},
    \qquad j\ge-1,\quad i\in\mathcal I_j;
\end{equation}
\item
\begin{equation}
    r_{j,i}
    =
    \frac18\operatorname{dist}(p_{j,i},S).
\end{equation}
\end{enumerate}
\end{Th}

\begin{proof}
The proof is the same as in \cite[Proposition~4.1.15]{sob_mms}, using Lemma~\ref{Lm.loc_d_m} to obtain the required bounded overlap.
\end{proof}
\par
We also fix a Lipschitz partition of unity \({\psi_{j,i}}\) subordinate to this Whitney decomposition; see \cite[p.~109]{sob_mms}. Thus,
\begin{enumerate}
\item \(\psi_{j,i}=0\) outside \(2B_{j,i}\), and \(\psi_{j,i}\ge c>0\) on \(B_{j,i}\);
\item
\begin{equation}
    \sum_{j\ge -1}\sum_{i\in\mathcal I_j}
    \psi_{j,i}(x)
    =
    1
\end{equation}
whenever \(\operatorname{dist}(x,S)\in(0,1)\);
\item
\begin{equation}
    L(\psi_{j,i})
    \lesssim
    r_{j,i}^{-1}.
\end{equation}
\end{enumerate}

Since \(X\) is proper, for every Whitney ball \(B_{j,i}\) we may choose a point \(\widehat p_{j,i}\in S\) such that
\begin{equation}
d(p_{j,i},\widehat p_{j,i})
=
\operatorname{dist}(p_{j,i},S).
\end{equation}
The choice of \(\widehat p_{j,i}\) is not necessarily unique, but this will not matter. We set
\begin{equation}
\label{eq.Uji_def}
U_{j,i}:=B_{r_{j,i}}(\widehat p_{j,i}).
\end{equation}
\par
For \(k\in\mathbb N\), define
\begin{equation}
\label{eq.Psi_k_def}
\Psi_k
:=
\sum_{j=k}^{\infty}
\sum_{i\in\mathcal I_j}
\psi_{j,i}.
\end{equation}
\par
The following elementary properties will be used repeatedly.
\begin{Lm}
\label{Lm.whit_prop}
Let \(S\subset X\) be closed, and let \(\{B_{j,i}\}\) be the Whitney decomposition from Theorem~\ref{Th.whit}. Then the following hold.
\begin{enumerate}
\item For every \(j\ge -1\),
\begin{equation}
\sum_{i\in\mathcal I_j}
\chi_{U_{j,i}}(x)
\le
C,
\qquad x\in S,
\end{equation}
where \(C\) is independent of \(j\).
\item Whenever
\begin{equation}
    2B_{j,i}\cap B_{l,m}\neq\emptyset,
\end{equation}
one has $|j-l|\le1$.
\item If
\begin{equation}
    J_{l,m}
    :=
    \{(j,i):2B_{j,i}\cap B_{l,m}\neq\emptyset\},
\end{equation}
then $|J_{l,m}|\le C$,
where \(C\) is independent of \(l\) and \(m\).
\item For every \(k\in\mathbb N\),
\begin{equation}
\label{eq.lip_Psi_k}
    \operatorname{lip}\Psi_k
    \lesssim
    2^k\chi_{G_k},
\end{equation}
where
\begin{equation}
\label{eq.Gk_def}
    G_k
    :=
    \bigcup_{\substack{j\ge -1\\ |j-k|\le1}}
    \bigcup_{i\in\mathcal I_j}
    2B_{j,i}.
\end{equation}
\end{enumerate}
\end{Lm}
\begin{proof}
\emph{Step 1.}
First, fix \(j\ge -1\) and \(x\in S\). Put
\begin{equation}
    I_j(x):=\{i\in\mathcal I_j:x\in U_{j,i}\}.
\end{equation}
We show that \(|I_j(x)|\le C\), uniformly in \(j\) and \(x\).
If \(i\in I_j(x)\), then \(x\in U_{j,i}=B_{r_{j,i}}(\widehat p_{j,i})\), and hence $d(x,\widehat p_{j,i})<r_{j,i}$.
Since $ d(p_{j,i},\widehat p_{j,i})
    =
    \operatorname{dist}(p_{j,i},S)
    =
    8r_{j,i}$,
we obtain
\begin{equation}
    d(p_{j,i},x)
    \le
    d(p_{j,i},\widehat p_{j,i})
    +
    d(\widehat p_{j,i},x)
    <
    9r_{j,i}
    \le
    9\cdot 2^{-j}.
\end{equation}
Therefore $x \in B_{9\cdot 2^{-j}}(p_{j, i}) \subset 20 B_{j, i}$. By the bounded overlap of Whitney balls (property (2) in Theorem~\ref{Th.whit}), we obtain the desired estimate for $|I_j(x)|$.
\par
\emph{Step 2.}
Assume that $2B_{j,i}\cap B_{l,m}\neq\emptyset$.
Then there exists \(z\in 2B_{j,i}\cap B_{l,m}\), and hence
$d(p_{j,i},p_{l,m})
    <
    2r_{j,i}+r_{l,m}$.
Using $\operatorname{dist}(p_{j,i},S)=8r_{j,i}$ and $\operatorname{dist}(p_{l,m},S)=8r_{l,m}$,
we obtain
\begin{equation}
    8|r_{j,i}-r_{l,m}|
    =
    \left|
    \operatorname{dist}(p_{j,i},S)
    -
    \operatorname{dist}(p_{l,m},S)
    \right|
    \le
    d(p_{j,i},p_{l,m})
    <
    2r_{j,i}+r_{l,m}.
\end{equation}
Considering separately the cases \(r_{j,i}\ge r_{l,m}\) and
\(r_{l,m}\ge r_{j,i}\), one obtains
\begin{equation}
    r_{j,i}< \frac32 r_{l,m},
    \qquad
    r_{l,m}< \frac{10}{7}r_{j,i}.
\end{equation}
Since
\begin{equation}
    2^{-j-1}<r_{j,i}\le2^{-j},
    \qquad
    2^{-l-1}<r_{l,m}\le2^{-l},
\end{equation}
this implies \(|j-l|\le1\).
\par
\emph{Step 3.}
The estimate \(|J_{l,m}|\le C\) now follows from the previous part and the bounded overlap of the Whitney balls. Indeed, all balls \(B_{j,i}\) with \((j,i)\in J_{l,m}\) have radii comparable to \(r_{l,m}\), their centers lie in a fixed enlargement of \(B_{l,m}\), and Lemma~\ref{Lm.loc_d_m} gives a uniform bound on their number.
\par
\emph{Step 4.} Since
\begin{equation}
    \Psi_k=\sum_{j=k}^{\infty}\sum_{i\in\mathcal I_j}\psi_{j,i},
\end{equation}
we have
\begin{equation}
    \operatorname{lip}\Psi_k(x)
    \le
    \sum_{j=k}^{\infty}\sum_{i\in\mathcal I_j}
    \operatorname{lip}\psi_{j,i}(x).
\end{equation}
If \(x\notin G_k\), then either no function \(\psi_{j,i}\) with \(j\ge k\) is nonzero near \(x\), or the sum is locally constant and equal to \(1\). Hence \(\operatorname{lip}\Psi_k(x)=0\). If \(x\in G_k\), then only indices with \(|j-k|\le1\) may contribute to the variation of \(\Psi_k\), and for them
\begin{equation}
    \operatorname{lip}\psi_{j,i}
    \lesssim
    r_{j,i}^{-1}
    \lesssim
    2^k.
\end{equation}
Using the bounded overlap of the supports \(2B_{j,i}\), we obtain
\begin{equation}
    \operatorname{lip}\Psi_k(x)
    \lesssim
    2^k\chi_{G_k}(x).
\end{equation}
\end{proof}
\par
Now set \(S=E\). For $\phi\in L_p(E,\mathcal H_\theta\lfloor_E)$
put
\begin{equation}
\label{eq.phi_ji_def}
    \phi_{j,i}
    :=
    \fint\limits_{U_{j,i}}
    \phi(x)\,d\mathcal H_\theta\lfloor_E(x).
\end{equation}
Let $\phi\in\operatorname{LIP}(E)\cap L_p(E,\mathcal H_\theta\lfloor_E)$
have bounded support. Choose \(\overline x\in E\) and \(R\ge1\) such that $\operatorname{supp}\phi\subset B_R(\overline x)$.
If \(\phi=0\) in \(L_p(E,\mathcal H_\theta\lfloor_E)\), we set $\operatorname{Ext}\phi=0$.
Otherwise, choose \(k=k(\phi)\in\mathbb N\) so large that
\begin{equation}
\label{eq.k_phi_choice}
    2^{-k}L(\phi)
    \left(
    \mathcal H_\theta\lfloor_E(B_{c_0R}(\overline x))
    \right)^{1/p}
    \le
    \|\phi\|_{L_p(E,\mathcal H_\theta\lfloor_E)},
\end{equation}
where \(c_0>1\) is a fixed geometric constant. Define
\begin{equation}
\label{eq.ext_lip_def}
    \operatorname{Ext}\phi(x)
    :=
    \sum_{j=k}^{\infty}
    \sum_{i\in\mathcal I_j}
    \phi_{j,i}\psi_{j,i}(x),
    \qquad x\in X\setminus E.
\end{equation}
By Lemma~\ref{Lm.AD_zero_mu}, we have $\mu(E)=0$. Hence the values of $\operatorname{Ext}\phi$
on $E$ do not affect any $L_p(X,\mu)$-estimates. However, for Lipschitz boundary
data we shall use the following representative:
\begin{equation}
    \operatorname{Ext}\phi(x)=\phi(x), \qquad \text{for all } x\in E.
\end{equation}
This choice will allow us to apply Remark~\ref{Rm.lip_est}.
\begin{Lm}
\label{Lm.lip_representative_ext}
Let $\phi\in \operatorname{LIP}(E)\cap L_p(E,\mathcal{H}_\theta\lfloor_E)$ have bounded support,
and let $\operatorname{Ext}\phi$ be defined by \eqref{eq.ext_lip_def} on $X\setminus E$ and by
$\operatorname{Ext}\phi=\phi$ on $E$. Then $\operatorname{Ext}\phi\in \operatorname{LIP}(X)$.
Moreover, if $f:=\operatorname{Ext}\phi$, then, for every $l\ge k-1$ and every
$m\in \mathcal{I}_l$,
\begin{equation}\label{eq:lip_ext_pointwise}
    \operatorname{lip} f(x)
    \lesssim
    \frac{1}{r_{l,m}}
    \sum_{(j,i)\in J_{l,m}}
    |\phi_{j,i}-\phi_{l,m}|
    +
    2^k |\phi_{l,m}|\chi_{G_k}(x),
    \qquad x\in B_{l,m}.
\end{equation}
\end{Lm}
\begin{proof}
Put $f:=\operatorname{Ext}\phi$. Since $\phi$ is Lipschitz and has bounded support,
it is bounded on $E$. The family $\{2B_{j,i}\}_{j,i}$ has uniformly bounded overlap, and
each function $\psi_{j,i}$ is Lipschitz with $L(\psi_{j,i})\lesssim r_{j,i}^{-1}$.
Thus the sum defining $f$ is locally finite on $X\setminus E$, and $f$ is locally Lipschitz
there.

Let $l\ge k-1$, $m\in \mathcal I_l$, and let $x,y\in B_{l,m}$. Only those functions $\psi_{j,i}$
whose supports meet $B_{l,m}$ may contribute to the difference $f(x)-f(y)$. Hence, by
the definition of $J_{l,m}$,
\begin{equation}
\begin{aligned}
|f(x)-f(y)|
&=
\left|
\sum_{(j,i)\in J_{l,m}}
\phi_{j,i}\bigl(\psi_{j,i}(x)-\psi_{j,i}(y)\bigr)
\right|  \\
&=
\left|
\sum_{(j,i)\in J_{l,m}}
(\phi_{j,i}-\phi_{l,m})
\bigl(\psi_{j,i}(x)-\psi_{j,i}(y)\bigr)
+
\phi_{l,m}\bigl(\Psi_k(x)-\Psi_k(y)\bigr)
\right|.
\end{aligned}
\end{equation}
Using the Lipschitz bound for $\psi_{j,i}$, the comparability $r_{j,i}\approx r_{l,m}$
for $(j,i)\in J_{l,m}$, and Lemma~\ref{Lm.whit_prop}, we get
\begin{equation}
    |f(x)-f(y)|
    \lesssim
    d(x,y)
    \left(
    \frac{1}{r_{l,m}}
    \sum_{(j,i)\in J_{l,m}}
    |\phi_{j,i}-\phi_{l,m}|
    +
    2^k |\phi_{l,m}|
    \right).
\end{equation}
Passing to the limit as $y\to x$ gives \eqref{eq:lip_ext_pointwise}, the factor
$\chi_{G_k}$ in the second term follows from the estimate
$\operatorname{lip}\Psi_k\lesssim 2^k\chi_{G_k}$.
\par
It remains only to justify the Lipschitz continuity across $E$. Let $\xi\in E$ and
$x\in X\setminus E$. If $d(x,E)$ is sufficiently small, then only Whitney balls of
levels $j\ge k$ may contribute to $f(x)$, and the partition of unity gives
\begin{equation}
    f(x)-\phi(\xi)
    =
    \sum_{j\ge k}\sum_{i\in \mathcal I_j}
    \bigl(\phi_{j,i}-\phi(\xi)\bigr)\psi_{j,i}(x).
\end{equation}
Whenever $\psi_{j,i}(x)\ne0$, the Whitney construction implies
\begin{equation}
    r_{j,i}\lesssim d(x,E)\le d(x,\xi),
    \qquad
    d(\widehat p_{j,i},\xi)\lesssim d(x,\xi).
\end{equation}
Since $U_{j,i}=B_{r_{j,i}}(\widehat p_{j,i})$, the Lipschitz continuity of $\phi$ gives
\begin{equation}
    |\phi_{j,i}-\phi(\xi)|
    \lesssim L(\phi)d(x,\xi).
\end{equation}
Using the bounded overlap of the partition, we obtain
\begin{equation}
    |f(x)-\phi(\xi)|\lesssim L(\phi)d(x,\xi).
\end{equation}
Therefore $f$ is Lipschitz across $E$.

Combining the Lipschitz continuity on $X\setminus E$ with the above boundary estimate,
we conclude that $f\in \operatorname{LIP}(X)$. Since $\mu(E)=0$, the pointwise estimate
\eqref{eq:lip_ext_pointwise} holds for $\mu$-almost every $x\in X$.
\end{proof}
\par
The integer \(k(\phi)\) is part of the construction. This dependence on \(\phi\) is the reason why the extension operator constructed below is nonlinear.
\begin{Lm}
\label{Lm.lip_ext_norm}
There exists \(C>0\) such that every Lipschitz function $\phi\in L_p(E, \mathcal{H}_{\theta}\lfloor_E)$
with bounded support satisfies
\begin{equation}
\label{eq.lip_ext_norm}
    \|\operatorname{Ext}\phi\|_{B^{\theta/p}_{p,1}(X)}
    \le
    C\|\phi\|_{L_p(E,\mathcal H_\theta\lfloor_E)}.
\end{equation}
\end{Lm}

\begin{proof}
Let \(\phi\) be a Lipschitz function with bounded support. Choose \(\overline x\in E\) and \(R\ge1\) such that $ \operatorname{supp}\phi\subset B_R(\overline x)$.
Let \(k=k(\phi)\) be chosen as in \eqref{eq.k_phi_choice}, and put $f:=\operatorname{Ext}\phi$.
\par
\emph{Step 1: the \(L_p(X)\)-estimate.}
We first estimate the \(L_p(X)\)-norm of \(f\). Since \(\operatorname{supp}\psi_{j,i}\subset 2B_{j,i}\) and the balls \(2B_{j,i}\) have uniformly bounded overlap, Jensen's inequality gives
\begin{equation}
\label{eq.tr_r_l_1}
\begin{split}
    \|f\|_{L_p(X)}^p
    &\lesssim
    \sum_{j=k}^{\infty}
    \sum_{i\in\mathcal I_j}
    |\phi_{j,i}|^p\mu(2B_{j,i}) \le
    \sum_{j=k}^{\infty}
    \sum_{i\in\mathcal I_j}
    \mu(2B_{j,i})
    \fint\limits_{U_{j,i}}
    |\phi(x)|^p\,d\mathcal H_\theta\lfloor_E(x).
\end{split}
\end{equation}
For every \(j\ge k\), the local doubling property of \(\mu\) implies
\begin{equation}
\label{eq.doubling_for_ext}
    \mu(2B_{j,i})
    =
    \mu(B_{2r_{j,i}}(p_{j,i}))
    \lesssim
    \mu(B_{r_{j,i}}(\widehat p_{j,i})).
\end{equation}
Indeed, \(d(p_{j,i},\widehat p_{j,i})=8r_{j,i}\), and hence $B_{2r_{j,i}}(p_{j,i})
    \subset
    B_{10r_{j,i}}(\widehat p_{j,i})$.
Thus \eqref{eq.doubling_for_ext} follows by applying the local doubling property a finite number of times. By the lower codimension-\(\theta\) regularity of \(\mathcal H_\theta\lfloor_E\), we have
\begin{equation}
\label{eq.phi_ji_mu_est}
\begin{split}
    \mu(2B_{j,i})
    \fint\limits_{U_{j,i}}
    |\phi(x)|^p\,d\mathcal H_\theta\lfloor_E(x)
    \lesssim
    r_{j,i}^{\theta}
    \int\limits_{U_{j,i}}
    |\phi(x)|^p\,d\mathcal H_\theta\lfloor_E(x).
\end{split}
\end{equation}
Therefore, using the bounded overlap of the family \(\{U_{j,i}\}_{i\in\mathcal I_j}\) for each fixed \(j\), we obtain
\begin{equation}
\label{eq.ext_Lp_est}
\begin{split}
    \|f\|_{L_p(X)}^p
    \lesssim
    \sum_{j=k}^{\infty}
    \sum_{i\in\mathcal I_j}
    2^{-j\theta}
    \int\limits_{U_{j,i}}
    |\phi(x)|^p\,d\mathcal H_\theta\lfloor_E(x) \\
    \lesssim
    \sum_{j=k}^{\infty}
    2^{-j\theta}
    \|\phi\|_{L_p(E,\mathcal H_\theta\lfloor_E)}^p \lesssim
    2^{-k\theta}
    \|\phi\|_{L_p(E,\mathcal H_\theta\lfloor_E)}^p.
\end{split}
\end{equation}
\par
\emph{Step 2: the \(L_p(X)\)-estimate for \(\operatorname{lip}f\).} By Lemma~\ref{Lm.lip_representative_ext}, for \(\mu\)-almost every \(x\in B_{l,m}\), \(l\ge k-1\), \(m\in\mathcal I_l\), we have \begin{equation} \label{eq.lip_ext_pointwise_used} \operatorname{lip} f(x) \lesssim \frac{1}{r_{l,m}} \sum_{(j,i)\in J_{l,m}} |\phi_{j,i}-\phi_{l,m}| + 2^k |\phi_{l,m}|\chi_{G_k}(x). \end{equation} Integrating \eqref{eq.lip_ext_pointwise_used} over \(X\), and using the bounded overlap of the Whitney balls, gives \begin{equation} \label{eq.tr_r_l_18} \begin{split} \|\operatorname{lip}f\|_{L_p(X)}^p &\lesssim \sum_{l=k-1}^{\infty} \sum_{m\in\mathcal I_l} \mu(B_{l,m}) \left( \frac1{r_{l,m}} \sum_{(j,i)\in J_{l,m}} |\phi_{j,i}-\phi_{l,m}| \right)^p \\ &\quad+ \sum_{\substack{l\ge k-1,\ m\in\mathcal I_l:\\ B_{l,m}\cap G_k\neq\emptyset}} \mu(B_{l,m}) \left( 2^k|\phi_{l,m}| \right)^p =: I_1+I_2. \end{split} \end{equation} \par
\par
\emph{Step 3: estimate of \(I_2\).}
We first estimate \(I_2\). Since \(B_{l,m}\cap G_k\neq\emptyset\) implies \(|l-k|\le C\), the same argument as in \eqref{eq.phi_ji_mu_est} gives
\begin{equation}
\label{eq.tr_r_l_19}
\begin{split}
    I_2
    \lesssim
    2^{kp}
    \sum_{|l-k|\le C}
    \sum_{m\in\mathcal I_l}
    \mu(B_{l,m})|\phi_{l,m}|^p &\lesssim
    2^{kp}
    \sum_{|l-k|\le C}
    \sum_{m\in\mathcal I_l}
    r_{l,m}^{\theta}
    \int\limits_{U_{l,m}}
    |\phi(x)|^p\,d\mathcal H_\theta\lfloor_E(x) \\&
    \lesssim
    2^{k(p-\theta)}
    \|\phi\|_{L_p(E,\mathcal H_\theta\lfloor_E)}^p.
\end{split}
\end{equation}
\par
\emph{Step 4: estimate of \(I_1\).}
It remains to estimate \(I_1\). Since \(\operatorname{supp}\phi\subset B_R(\overline x)\), only those balls \(B_{l,m}\) for which \(U_{l,m}\) meets a fixed enlargement of \(B_R(\overline x)\) contribute to \(I_1\). More precisely, there exists a constant \(c_1>1\), independent of \(\phi\), such that it is enough to sum over those \(m\in\mathcal I_l\) for which
\begin{equation}
    U_{l,m}\cap B_{c_1R}(\overline x)\neq\emptyset.
\end{equation}
Fix \(l\ge k-1\). By the lower codimension-\(\theta\) regularity and the bounded overlap of the family \(\{U_{l,m}\}_{m\in\mathcal I_l}\),
\begin{equation}
\label{eq.tr_rev_new1}
\begin{split}
    \sum_{\substack{m\in\mathcal I_l:\\ U_{l,m}\cap B_{cR}(\overline x)\neq\emptyset}}
    \frac{\mu(B_{l,m})}{r_{l,m}^{\theta}}
    \lesssim
    \sum_{\substack{m\in\mathcal I_l:\\ U_{l,m}\cap B_{cR}(\overline x)\neq\emptyset}}
    \mathcal H_\theta\lfloor_E(U_{l,m})
    \lesssim
    \mathcal H_\theta\lfloor_E(B_{\widetilde cR}(\overline x)).
\end{split}
\end{equation}

Let \((j,i)\in J_{l,m}\). Then \(r_{j,i}\approx r_{l,m}\), and the sets \(U_{j,i}\) and \(U_{l,m}\) lie within distance \(Cr_{l,m}\) from one another. Hence, since \(\phi\) is Lipschitz,
\begin{equation}
\label{eq.average_difference_lip}
    |\phi_{j,i}-\phi_{l,m}|
    \lesssim
    r_{l,m}L(\phi).
\end{equation}
Using \eqref{eq.average_difference_lip}, the uniform bound for \(|J_{l,m}|\), and \eqref{eq.tr_rev_new1}, we obtain
\begin{equation}
\label{eq.tr_r_l_23}
\begin{split}
    I_1
    \lesssim
    \sum_{l=k-1}^{\infty}
    \sum_{\substack{m\in\mathcal I_l:\\ U_{l,m}\cap B_{cR}(\overline x)\neq\emptyset}}
    \mu(B_{l,m})&L(\phi)^p 
    \lesssim
    L(\phi)^p
    \sum_{l=k-1}^{\infty}
    2^{-l\theta}
    \mathcal H_\theta\lfloor_E(B_{\widetilde cR}(\overline x)) \\
    &\lesssim
    2^{-k\theta}
    L(\phi)^p
    \mathcal H_\theta\lfloor_E(B_{\widetilde cR}(\overline x)).
\end{split}
\end{equation}
\par
Combining \eqref{eq.tr_r_l_18}, \eqref{eq.tr_r_l_19}, and \eqref{eq.tr_r_l_23}, we get
\begin{equation}
\label{eq.lip_ext_est_full}
    \|\operatorname{lip}f\|_{L_p(X)}^p
    \lesssim
    2^{k(p-\theta)}
    \|\phi\|_{L_p(E,\mathcal H_\theta\lfloor_E)}^p
    +
    2^{-k\theta}
    L(\phi)^p
    \mathcal H_\theta\lfloor_E(B_{\widetilde cR}(\overline x)).
\end{equation}
Consequently,
\begin{equation}
\label{eq.scaled_lip_est}
    2^{-k(1-\frac{\theta}{p})}
    \|\operatorname{lip}f\|_{L_p(X)}
    \lesssim
    \|\phi\|_{L_p(E,\mathcal H_\theta\lfloor_E)}
    +
    2^{-k}L(\phi)
    \left(
    \mathcal H_\theta\lfloor_E(B_{\widetilde cR}(\overline x))
    \right)^{1/p}.
\end{equation}
By choosing the constant \(c_0\) in \eqref{eq.k_phi_choice} large enough ($c_0\ge \widetilde c$), the second term on the right-hand side of \eqref{eq.scaled_lip_est} is bounded by $\|\phi\|_{L_p(E,\mathcal H_\theta\lfloor_E)}$.
Therefore
\begin{equation}
\label{eq.scaled_lip_final}
    2^{-k(1-\frac{\theta}{p})}
    \|\operatorname{lip}f\|_{L_p(X)}
    \lesssim
    \|\phi\|_{L_p(E,\mathcal H_\theta\lfloor_E)}.
\end{equation}
\par
\emph{Step 5.}
By Lemma~\ref{Lm.lip_representative_ext}, we have \(f\in\operatorname{LIP}(X)\).
Therefore Remark~\ref{Rm.lip_est} can be applied with \(s=\frac{\theta}{p}\) and
\(\delta=2^{-k}\). Using \eqref{eq.ext_Lp_est} and \eqref{eq.scaled_lip_final}, we get
\begin{equation}
    \|f\|_{b^{\theta/p}_{p,1}(X)}
    \lesssim
    2^{-k(1-\frac{\theta}{p})}\|\operatorname{lip}f\|_{L_p(X)}
    +
    2^{k\frac{\theta}{p}}\|f\|_{L_p(X)}
    \lesssim
    \|\phi\|_{L_p(E,\mathcal H_\theta\lfloor_E)}.
\end{equation}
Moreover, \eqref{eq.ext_Lp_est} implies
\begin{equation}
    \|f\|_{L_p(X)}
    \lesssim
    2^{-k\frac{\theta}{p}}\|\phi\|_{L_p(E,\mathcal H_\theta\lfloor_E)}
    \lesssim
    \|\phi\|_{L_p(E,\mathcal H_\theta\lfloor_E)}.
\end{equation}
Combining the last two estimates gives \eqref{eq.lip_ext_norm}.
\end{proof}

\begin{Lm}
\label{Lm.ext_lip_ball}
Let $\phi\in\operatorname{LIP}(E)\cap L_p(E,\mathcal H_\theta\lfloor_E)$ have
bounded support, and let \(f:=\operatorname{Ext}\phi\) be defined by
\eqref{eq.ext_lip_def}. Then there exist constants \(c>1\) and \(C>0\), independent
of \(\phi\), such that for every \(x\in E\), every \(a\in\mathbb R\), and every
\(r>0\) sufficiently small,
\begin{equation}
\label{eq.ext_lip_ball}
    \int\limits_{B_r(x)}
    |f(y)-a|^p\,d\mu(y)
    \le
    C r^\theta
    \int\limits_{B_{cr}(x)}
    |\phi(y)-a|^p\,d\mathcal H_\theta\lfloor_E(y).
\end{equation}
\end{Lm}

\begin{proof}
Since \(\mu(E)=0\), it is enough to estimate the integral over \(B_r(x)\setminus E\).
We choose \(r>0\) so small that all Whitney balls meeting \(B_r(x)\) have level
\(j\ge k(\phi)\). Then, on \(B_r(x)\setminus E\), the corresponding partition of
unity satisfies
\begin{equation}
    \sum_{j=k(\phi)}^\infty\sum_{i\in\mathcal I_j}\psi_{j,i}(y)=1.
\end{equation}
Hence, for every \(y\in B_r(x)\setminus E\),
\begin{equation}
\begin{aligned}
    f(y)-a
    &=
    \sum_{j=k(\phi)}^\infty\sum_{i\in\mathcal I_j}
    \phi_{j,i}\psi_{j,i}(y)
    -
    a\sum_{j=k(\phi)}^\infty\sum_{i\in\mathcal I_j}\psi_{j,i}(y) \\
    &=
    \sum_{j=k(\phi)}^\infty\sum_{i\in\mathcal I_j}
    (\phi_{j,i}-a)\psi_{j,i}(y).
\end{aligned}
\end{equation}
Moreover,
\begin{equation}
    \phi_{j,i}-a
    =
    \fint\limits_{U_{j,i}}
    (\phi(y)-a)\,d\mathcal H_\theta\lfloor_E(y).
\end{equation}
\par
Let \(\mathcal A_r(x)\) denote the family of indices \((j,i)\) such that
 $2B_{j,i}\cap B_r(x)\neq\emptyset$.
By Jensen's inequality and the bounded overlap of the family \(\{2B_{j,i}\}_{j,i}\),
we obtain
\begin{equation}
\begin{aligned}
    \int\limits_{B_r(x)}
    |f(y)-a|^p\,d\mu(y)
    &\lesssim
    \sum_{(j,i)\in\mathcal A_r(x)}
    |\phi_{j,i}-a|^p\,\mu(2B_{j,i})                              \\
    &\le
    \sum_{(j,i)\in\mathcal A_r(x)}
    \mu(2B_{j,i})
    \fint\limits_{U_{j,i}}
    |\phi(y)-a|^p\,d\mathcal H_\theta\lfloor_E(y).
\end{aligned}
\end{equation}
By the local doubling property of \(\mu\) and the lower codimension-\(\theta\)
regularity of \(\mathcal H_\theta\lfloor_E\),
\begin{equation}
    \mu(2B_{j,i})
    \lesssim
    r_{j,i}^{\theta}(\mathcal H_\theta\lfloor_E)(U_{j,i}).
\end{equation}
If \((j,i)\in\mathcal A_r(x)\), then the Whitney construction implies
\(r_{j,i}\lesssim r\) and $U_{j,i}\subset B_{cr}(x)$
with a structural constant \(c>1\). Therefore
\begin{equation}
\begin{aligned}
    \int\limits_{B_r(x)}
    |f(y)-a|^p\,d\mu(y)
    &\lesssim
    r^\theta
    \sum_{(j,i)\in\mathcal A_r(x)}
    \int\limits_{U_{j,i}}
    |\phi(y)-a|^p\,d\mathcal H_\theta\lfloor_E(y) \\
    &\lesssim
    r^\theta
    \int\limits_{B_{cr}(x)}
    |\phi(y)-a|^p\,d\mathcal H_\theta\lfloor_E(y),
\end{aligned}
\end{equation}
where in the last step we used the bounded overlap of the family
\(\{U_{j,i}\}_{j,i}\). This proves \eqref{eq.ext_lip_ball}.
\end{proof}

\begin{Lm}
\label{Lm.lip_leb_point}
Let \(\phi\in\operatorname{LIP}(E)\cap L_p(E,\mathcal H_\theta\lfloor_E)\) have bounded support, and let \(f:=\operatorname{Ext}\phi\) be defined by \eqref{eq.ext_lip_def}. Then
\begin{equation}
    \lim_{r\to0}
    \fint\limits_{B_r(x)}
    |f(y)-\phi(x)|^p\,d\mu(y)
    =
    0
\end{equation}
for every \(x\in E\). In particular, $\operatorname{Tr}(\operatorname{Ext}\phi) = \phi$.
\end{Lm}

\begin{proof}
For \(r>0\) sufficiently small, say \(r<c2^{-k(\phi)}\), all Whitney balls meeting \(B_r(x)\) have level \(j\ge k(\phi)\). Therefore
\begin{equation}
    \sum_{j=k(\phi)}^\infty\sum_{i\in\mathcal{I}_j}\psi_{j,i}(y)=1
\end{equation}
for every \(y\in B_r(x)\setminus E\). Since \(\mu(E)=0\) by Lemma~\ref{Lm.AD_zero_mu}, this is sufficient for the integral estimates below.
Consequently,
\begin{equation}
    |f(y)-\phi(x)|
    \le
    \sum_{j=k(\phi)}^\infty
    \sum_{i\in\mathcal{I}_j}
    |\phi_{j,i}-\phi(x)|\psi_{j,i}(y).
\end{equation}
Applying Lemma~\ref{Lm.ext_lip_ball} with \(a=\phi(x)\), we get
\begin{equation}
    \int\limits_{B_r(x)}
    |f(y)-\phi(x)|^p\,d\mu(y)
    \lesssim
    r^\theta
    \int\limits_{B_{cr}(x)}
    |\phi(y)-\phi(x)|^p\,d\mathcal H_\theta\lfloor_E(y).
\end{equation}
Dividing by $\mu(B_r(x))$, using the upper codimension-$\theta$ regularity of
$\mathcal H_\theta\lfloor_E$ together with the local doubling property of $\mu$,
and then the Lipschitz continuity of $\phi$, we obtain
\begin{equation}
    \fint\limits_{B_r(x)}
    |f(y)-\phi(x)|^p\,d\mu(y)
    \lesssim
    \fint\limits_{B_{cr}(x)}
    |\phi(y)-\phi(x)|^p\,d\mathcal H_\theta\lfloor_E(y) \le L(\phi)^pr^p
    \to0
\end{equation}
as \(r\to0\). This proves the claim.
\end{proof}
\begin{Th}
\label{Th.tr_r}
There exists a bounded nonlinear extension operator
\begin{equation}
    \operatorname{Ext}:
    L_p(E,\mathcal H_\theta\lfloor_E)
    \rightarrow
    B^{\theta/p}_{p,1}(X),
\end{equation}
which is a right inverse to the trace operator.
\end{Th}

\begin{proof}
Let \(\phi\in L_p(E,\mathcal H_\theta\lfloor_E)\). If \(\phi=0\) in
\(L_p(E,\mathcal H_\theta\lfloor_E)\), we set \(\operatorname{Ext}\phi=0\).
Assume now that \(\phi\ne0\). Since \((E,d,\mathcal H_\theta\lfloor_E)\) is proper and \(\mathcal H_\theta\lfloor_E\) is locally finite and Borel regular, boundedly supported Lipschitz functions are dense in \(L_p(E,\mathcal H_\theta\lfloor_E)\). Hence we can choose a sequence \(\{\phi_l\}_{l=1}^{\infty}
    \subset
    \operatorname{LIP}(E)\cap L_p(E,\mathcal H_\theta\lfloor_E)\)
of boundedly supported functions such that, with \(\phi_0:=0\) and $u_l:=\phi_l-\phi_{l-1}$,
one has
\begin{equation}
\label{eq.approx_sequence_explicit}
    \|u_l\|_{L_p(E,\mathcal H_\theta\lfloor_E)}
    \le
    C2^{-l}
    \|\phi\|_{L_p(E,\mathcal H_\theta\lfloor_E)},
    \qquad l\in\mathbb N,
\end{equation}
and therefore
\begin{equation}
\label{eq.tr_r_g1}
    \phi
    =
    \sum_{l=1}^{\infty}u_l
\end{equation}
both in \(L_p(E,\mathcal H_\theta\lfloor_E)\) and
\(\mathcal H_\theta\lfloor_E\)-almost everywhere.

Indeed, the convergence in \(L_p\) follows from the construction, and
\begin{equation}
    \sum_{l=1}^{\infty}
    \|u_l\|_{L_p(E,\mathcal H_\theta\lfloor_E)}
    \lesssim
    \|\phi\|_{L_p(E,\mathcal H_\theta\lfloor_E)}.
\end{equation}
Therefore
\begin{equation}
    \sum_{l=1}^{\infty}|u_l|
    \in
    L_p(E,\mathcal H_\theta\lfloor_E),
\end{equation}
and hence the series in \eqref{eq.tr_r_g1} converges absolutely
\((\mathcal H_\theta\lfloor_E)\)-almost everywhere.
\par
\emph{Step 1: construction of the extension.}
For every \(l\in\mathbb N\), apply the previous construction to \(u_l\). Thus we choose
\(k_l=k(u_l)\in\mathbb N\) so that \eqref{eq.k_phi_choice} holds with \(u_l\) in place of
\(\phi\). Increasing \(k_l\), if necessary, we also assume that
\begin{equation}
\label{eq.kl_extra_condition_explicit}
    2^{-k_l\frac{\theta}{p}}L(u_l)
    \le
    2^{-l}
    \|\phi\|_{L_p(E,\mathcal H_\theta\lfloor_E)},
    \qquad l\in\mathbb N,
\end{equation}
and that the sequence \(\{k_l\}_{l=1}^{\infty}\) is strictly increasing.

We define
\begin{equation}
\label{eq.ext_general_def}
    \operatorname{Ext}\phi
    :=
    \sum_{l=1}^{\infty}
    \operatorname{Ext}u_l.
\end{equation}
\par
\emph{Step 2: boundedness.}
By Lemma~\ref{Lm.lip_ext_norm} and \eqref{eq.approx_sequence_explicit},
\begin{equation}
\begin{split}
    \sum_{l=1}^{\infty}
    \|\operatorname{Ext}u_l\|_{B^{\theta/p}_{p,1}(X)}
    \lesssim
    \sum_{l=1}^{\infty}
    \|u_l\|_{L_p(E,\mathcal H_\theta\lfloor_E)}
    \lesssim
    \|\phi\|_{L_p(E,\mathcal H_\theta\lfloor_E)}.
\end{split}
\end{equation}
Hence the series in \eqref{eq.ext_general_def} converges in
\(B^{\theta/p}_{p,1}(X)\), and
\begin{equation}
\label{eq.ext_general_bound}
    \|\operatorname{Ext}\phi\|_{B^{\theta/p}_{p,1}(X)}
    \lesssim
    \|\phi\|_{L_p(E,\mathcal H_\theta\lfloor_E)}.
\end{equation}
\par
\emph{Step 3: verification of the trace identity.}
It remains to prove that $\operatorname{Tr}(\operatorname{Ext}\phi)=\phi$
in \(L_p(E,\mathcal H_\theta\lfloor_E)\).
For \(N\in\mathbb N\), put
\begin{equation}
\label{eq.partial_ext_sum}
    F_N
    :=
    \sum_{l=1}^{N}
    \operatorname{Ext}u_l.
\end{equation}
Since each \(u_l\) is Lipschitz with bounded support, Lemma~\ref{Lm.lip_leb_point}
implies that $\operatorname{Tr}(\operatorname{Ext}u_l)=u_l$
in \(L_p(E,\mathcal H_\theta\lfloor_E)\). Hence, by the linearity of the trace
operator,
\begin{equation}
\label{eq.trace_partial_sum}
    \operatorname{Tr}F_N
    =
    \sum_{l=1}^{N}u_l
    =
    \phi_N
\end{equation}
in \(L_p(E,\mathcal H_\theta\lfloor_E)\).

By Step~2, the series defining \(\operatorname{Ext}\phi\) converges in
\(B^{\theta/p}_{p,1}(X)\). Therefore
\begin{equation}
    F_N\to \operatorname{Ext}\phi
    \qquad\text{in }B^{\theta/p}_{p,1}(X).
\end{equation}
Since \(E\) is Ahlfors--David codimension-\(\theta\) regular, the measure
\(\mathcal H_\theta\lfloor_E\) is upper codimension-\(\theta\) regular. Hence, by
Theorem~\ref{Th.dir}, the trace operator
\begin{equation}
    \operatorname{Tr}:B^{\theta/p}_{p,1}(X)
    \to
    L_p(E,\mathcal H_\theta\lfloor_E)
\end{equation}
is bounded. Consequently,
\begin{equation}
\label{eq.trace_partial_convergence}
    \operatorname{Tr}F_N
    \to
    \operatorname{Tr}(\operatorname{Ext}\phi)
    \qquad\text{in }L_p(E,\mathcal H_\theta\lfloor_E).
\end{equation}
On the other hand, by the construction of the approximating sequence,
\begin{equation}
    \phi_N=\sum_{l=1}^{N}u_l\to \phi
    \qquad\text{in }L_p(E,\mathcal H_\theta\lfloor_E).
\end{equation}
Combining this with \eqref{eq.trace_partial_sum} and
\eqref{eq.trace_partial_convergence}, we obtain $\operatorname{Tr}(\operatorname{Ext}\phi)=\phi$
in \(L_p(E,\mathcal H_\theta\lfloor_E)\). This proves that
\(\operatorname{Ext}\) is a right inverse to the trace operator.
\end{proof}
\section{Traces of Besov spaces on $K$-regular trees}
In this section, we discuss the trace problem on \(K\)-regular trees. The main
boundedness and extension results follow from the trace theorems obtained above in the
general setting of metric measure spaces. At the same time, the tree structure allows one
to formulate several additional properties and criteria in more explicit terms.
\subsection{Preliminaries}
We briefly recall the terminology related to \(K\)-regular trees. For background and
related results, we refer the reader to
\cite{shanm_tr_s, Ext_tr_h_f, tr_den_tr, dyad_norm} and the references therein.

A graph \(G\) is a pair \((V,\mathcal E)\), where \(V\) is the set of vertices and
\(\mathcal E\) is the set of edges. Two vertices \(x,y\in V\) are called neighbors if
they are connected by an edge. We turn \(G\) into a metric graph by identifying each edge
with an isometric copy of the unit interval. We denote the corresponding one-dimensional
length measure by \(\ell_G\).

A tree is a connected graph without cycles. A rooted tree is a tree with a distinguished
vertex, called the root, which we denote by \(0\). If \(x\in V\setminus\{0\}\), then the
unique neighbor of \(x\) lying closer to the root is called the parent of \(x\), the
remaining neighbors are called the children of \(x\). For the root, all neighbors are
called children.

A \(K\)-regular tree, $K\ge 1$, is a rooted tree in which every vertex has exactly \(K\) children.
Throughout this section, we denote by \(X\) the metric graph associated
with a fixed \(K\)-regular tree. Slightly abusing notation, we shall also call \(X\) a
\(K\)-regular tree.

For \(x,y\in X\), let \([x,y]\) denote the unique geodesic segment joining \(x\) and
\(y\). We write $d_G(x,y):=\ell_G([x,y])$
for the standard graph distance, and we also use the notation $|x-y|:=d_G(x,y)$.
In particular, $|x|:=d_G(0,x)$
denotes the graph distance from the root to \(x\).

An infinite geodesic ray in \(X\) is an isometric embedding $\gamma:[0,\infty)\to X$.
If \(\gamma(0)=0\), we say that \(\gamma\) is an infinite geodesic ray starting from the
root.

\par
Assume that \(w,\lambda\in L_1^{\operatorname{loc}}([0,\infty))\) are strictly
positive functions. We equip \(X\) with the measure
\begin{equation}
\label{eq.m_d}
    d\mu(x)=w(|x|)\,d\ell_G(x),
\end{equation}
where \(\ell_G\) denotes the one-dimensional length measure on the metric graph.
We also equip \(X\) with the path metric \(d_\lambda\) defined by
\begin{equation}
\label{eq.d_lambda_def}
    d_\lambda(x,y)
    =
    \int\limits_{[x,y]}\lambda(|z|)\,d\ell_G(z),
    \qquad x,y\in X.
\end{equation}

Let $j(t)=[t]+1$, $t\ge0$.
Then \(K^{j(t)}\) is the number of edges intersecting the level set
\(\{x\in X: |x|=t\}\), for a.e. \(t>0\). Hence the total measure of \(X\) is given by
\begin{equation}
\label{eq.total_measure_tree}
    \mu(X)
    =
    \int\limits_0^\infty w(t)K^{j(t)}\,dt.
\end{equation}
Throughout this section we assume that \(X\) has finite \(\mu\)-measure and finite
\(d_\lambda\)-diameter. Equivalently, we assume that
\begin{equation}
\label{eq.finit}
    \int\limits_0^\infty w(t)K^{j(t)}\,dt<\infty,
    \qquad
    \int\limits_0^\infty \lambda(t)\,dt<\infty.
\end{equation}

Let \(\overline X\) be the completion of \((X,d_\lambda)\). Since
\(\int\limits_0^\infty\lambda(t)\,dt<\infty\), each infinite geodesic ray starting from the
root defines a point of \(\overline X\). We set $\partial X:=\overline X\setminus X$.
Then \(\partial X\) can be identified with the family of infinite geodesic rays starting
from the root.

We do not regard \(X\) itself as a metric measure space in the sense of
Section~2, since \((X,d_\lambda)\) is not complete. Whenever we use the general trace
theorems from the previous sections, the ambient space will be the completion
\((\overline X,d_\lambda)\), and the measure \(\mu\) will be understood as extended to
\(\overline X\) by \(\mu(\partial X)=0\).

Unless explicitly stated otherwise, doubling assumptions or assumptions of the weak Poincar\'e inequality in this section refer to the
metric graph \(X\) equipped with the metric \(d_\lambda\) and the measure $\mu$. The required properties
for balls centered at boundary points will be derived below from the doubling property on
\(X\). Furthermore, since we assume that $\operatorname{diam}(X)<\infty$, the local doubling property and the local weak Poincar\'e inequality are equivalent to their global counterparts. Therefore, we do not distinguish those properties throughout this section.

\begin{Example}[\cite{shanm_tr_s}] 
\label{Ex.st_m_d}
Let
\begin{equation}
\lambda(t)=e^{-\varepsilon t}, \qquad w(t)=e^{-\beta t},
\end{equation} 
where \(\varepsilon>0\) and \(\beta>\log K\). Then \(\mu(X)<\infty\) and \(\operatorname{diam}_\lambda(X)<\infty\). Moreover, the measure \(\mu\) is doubling, and \(( X,d_\lambda,\mu)\) supports a weak \((1,1)\)-Poincar\'e inequality. Consequently, it supports a weak \((1,p)\)-Poincar\'e inequality for every \(p\in[1,\infty)\).
\end{Example}
\par
For distinct \(\xi,\eta\in\partial X\), let \(x_{(\xi,\eta)}\) denote the last common
point of the geodesic rays \([0,\xi)\) and \([0,\eta)\). The restriction of the metric
\(d_\lambda\) to \(\partial X\) is given by
\begin{equation}
\label{eq.boundary_metric}
    d_b(\xi,\eta)
    =
    2\int\limits_{|x_{(\xi,\eta)}|}^{\infty}\lambda(t)\,dt,
    \qquad \xi\ne\eta,
\end{equation}
and we set \(d_b(\xi,\xi)=0\).
\par
For \(x\in X\), we denote by \(\Gamma_x\) the subtree rooted at \(x\), that is,
\begin{equation}
\label{eq.subtree_def}
    \Gamma_x
    :=
    \{y\in X: x\in[0,y]\}.
\end{equation}
We also write
\begin{equation}
    \partial\Gamma_x
    :=
    \{\xi\in\partial X: x\in[0,\xi)\}.
\end{equation}
For \(n\in\mathbb N_0\), let $V^n:=\{v\in V: |v|=n\}$
be the set of vertices at level \(n\). If \(v\in V^n\), set
\begin{equation}
\label{eq.cylinder_def}
    I_v
    :=
    \{\xi\in\partial X: v\in[0,\xi)\}.
\end{equation}
The natural probability measure \(\nu\) on \(\partial X\) is determined by
\begin{equation}
\label{eq.boundary_measure_cylinders}
    \nu(I_v)=K^{-n},
    \qquad v\in V^n.
\end{equation}
For \(K=1\), the boundary consists of a single point and the above formula simply gives
\(\nu(\partial X)=1\).

For further properties of \((\partial X,d_b)\), see \cite{dif_tr_op}. In particular,
\((\partial X,d_b)\) is an ultrametric space, and the Lebesgue differentiation theorem
holds for the metric measure space \((\partial X,d_b,\nu)\).

The following disintegration formula will be used below.

\begin{Lm}[\cite{dif_tr_op}, Lemma 3.2]
\label{Lm.subtr_comp}
Let \(X\) be a \(K\)-regular tree, and let
\(w,\lambda:[0,\infty)\to(0,\infty)\) be locally integrable. Let
\(p\in[1,\infty)\), and let \(f\in L_p(X)\). Assume that \(\mu(X)<\infty\).
Then, for every \(x\in X\),
\begin{equation}
\label{eq.subtree_disintegration}
    \int\limits_{\Gamma_x}|f(y)|^p\,d\mu(y)
    =
    \int\limits_{\partial\Gamma_x}
    \int\limits_{[x,\xi)}
    |f(y)|^p K^{j(|y|)}w(|y|)\,d\ell_G(y)\,d\nu(\xi).
\end{equation}
\end{Lm}

\begin{Lm}
\label{Lm.completion_admissibility}
Let \(X\) be a weighted \(K\)-regular tree satisfying \eqref{eq.finit}, and extend
\(\mu\) to \(\overline X\) by setting \(\mu(\partial X)=0\).
Assume that \(\mu\) is doubling on \((X,d_\lambda)\). Then the extended measure
\(\mu\) is doubling on \((\overline X,d_\lambda)\).
If, in addition, \((X,d_\lambda,\mu)\) supports a weak \((1,p)\)-Poincar\'e inequality
for some \(p\in[1,\infty)\), then \((\overline X,d_\lambda,\mu)\) supports a
weak \((1,p)\)-Poincar\'e inequality. In particular, $(\overline X,d_\lambda,\mu)\in\mathfrak U_p^{\operatorname{gl}}$.
\end{Lm}

\begin{proof}
We first prove the doubling property on the completion. Since
\(\mu(\partial X)=0\), balls centered at points of \(X\) have the same measure whether
they are considered in \(X\) or in \(\overline X\). Thus it remains to consider balls
centered at boundary points.

Fix \(\xi\in\partial X\). For \(a\in (0, \frac{\operatorname{diam}(X)}{4}]\), let
\(x_a=x_a(\xi)\in[0,\xi)\) be the unique point satisfying
\begin{equation}
\label{eq.boundary_point_xa}
    \int\limits_{|x_a|}^{\infty}\lambda(t)\,dt=a.
\end{equation}
Then
\begin{equation}
\label{eq.boundary_ball_subtree_inclusions}
    \Gamma_{x_{a/2}}
    \subset
    B_a^{\overline X}(\xi)\cap X
    \subset
    \Gamma_{x_a}.
\end{equation}
Moreover,
\begin{equation}
\label{eq.subtree_ball_comparison_completion}
    B_{a/4}^{X}(x_{a/4})
    \subset
    \Gamma_{x_{a/2}}
    \subset
    \Gamma_{x_a}
    \subset
    B_{7a/4}^{X}(x_{a/4}).
\end{equation}
By the doubling property of \(\mu\) on \(X\), \eqref{eq.subtree_ball_comparison_completion}
implies
\begin{equation}
\label{eq.subtree_doubling_completion}
    \mu(\Gamma_{x_a})
    \lesssim
    \mu(\Gamma_{x_{a/2}}).
\end{equation}
Applying this estimate twice and using \eqref{eq.boundary_ball_subtree_inclusions}, we get
\begin{equation}
    \mu\bigl(B_{2a}^{\overline X}(\xi)\bigr)
    \le
    \mu(\Gamma_{x_{2a}})
    \lesssim
    \mu(\Gamma_{x_{a/2}})
    \le
    \mu\bigl(B_a^{\overline X}(\xi)\bigr).
\end{equation}
Thus the doubling estimate holds for balls centered at boundary points.

It remains to prove the weak \((1,p)\)-Poincar\'e inequality on \(\overline X\).
Let \(f\in\operatorname{LIP}(\overline X)\). Since \(\mu(\partial X)=0\), all integrals
are taken over \(X\). Moreover, for every \(x\in X\),
\begin{equation}
    \operatorname{lip}_{\overline X}f(x)
    =
    \operatorname{lip}_{X}(f|_X)(x).
\end{equation}
Indeed, every point of \(X\) has positive \(d_\lambda\)-distance from \(\partial X\), and
therefore the local behavior of \(f\) near such a point is completely determined inside
\(X\).
Thus the weak \((1,p)\)-Poincar\'e inequality on \(\overline X\) follows immediately from
the corresponding inequality on \(X\) for balls centered at points of \(X\). It remains only
to consider balls centered at boundary points.

Fix \(\xi\in\partial X\) and \(a>0\). Let \(x_{a/4}=x_{a/4}(\xi)\in[0,\xi)\). By the
geometric inclusions proved above, $B_a^{\overline X}(\xi)\cap X
    \subset
    B_{\frac{7a}{4}}^{X}(x_{a/4})$,
and, by the doubling estimate on the completion already proved,
\begin{equation}
    \mu\bigl(B_{\frac{7a}{4}}^{X}(x_{a/4})\bigr)
    \lesssim
    \mu\bigl(B_a^{\overline X}(\xi)\bigr).
\end{equation}
Hence, using the elementary estimate for nested sets of comparable measure,
\begin{equation}
    \fint\limits_{B_a^{\overline X}(\xi)}
    \left|f(x)-A^{\mu}_af(\xi)d\mu\right|\,d\mu(x)
    \lesssim
    \fint\limits_{B_{\frac{7a}{4}}^{X}(x_{a/4})}
    \left|f(x)-A^{\mu}_{{\frac{7a}{4}}}f(x_{a/4})d\mu\right|\,d\mu(x) .
\end{equation}
Applying the weak \((1,p)\)-Poincar\'e inequality on \(X\) to the ball
\(B_{\frac{7a}{4}}^{X}(x_{a/4})\), and using the inclusion $B_{Ca}^{X}(x_{a/4})
    \subset
    B_{C'a}^{\overline X}(\xi)$, we obtain the desired result.
\end{proof}

For \(r\ge0\), denote by \(X^r\) the tail
\begin{equation}
\label{eq.tail_def}
    X^r
    :=
    \{x\in X: |x|\ge r\}.
\end{equation}
If \(0\le r<r'\), denote by \(X^{r,r'}\) the open strip
\begin{equation}
\label{eq.strip_def}
    X^{r,r'}
    :=
    \{x\in X: r<|x|<r'\}.
\end{equation}
We define the \(d_\lambda\)-width of \(X^{r,r'}\) by
\begin{equation}
\label{eq.strip_width_def}
    \mathcal W(X^{r,r'})
    :=
    \int\limits_r^{r'}\lambda(t)\,dt.
\end{equation}
The center level \(r_c(X^{r,r'})\) is the unique number in \((r,r')\) such that
\begin{equation}
\label{eq.strip_center_def}
    \int\limits_r^{r_c(X^{r,r'})}\lambda(t)\,dt
    =
    \int\limits_{r_c(X^{r,r'})}^{r'}\lambda(t)\,dt.
\end{equation}

Let \(c\in(0,1)\). We denote by \(cX^{r,r'}\) the strip
\(X^{r_c^-,r_c^+}\), where \(r_c^-<r_c(X^{r,r'})<r_c^+\) are chosen so that
\begin{equation}
\label{eq.centered_substrip_def}
    \int\limits_{r_c^-}^{r_c(X^{r,r'})}\lambda(t)\,dt
    =
    \int\limits_{r_c(X^{r,r'})}^{r_c^+}\lambda(t)\,dt
    =
    \frac{c}{2}\mathcal W(X^{r,r'}).
\end{equation}
Thus \(cX^{r,r'}\) is the strip with the same center level as \(X^{r,r'}\) and with
\(d_\lambda\)-width \(c\mathcal W(X^{r,r'})\).

For \(x\in X\setminus\{0\}\), let \(x^{(1)}\) denote the closest vertex to \(x\) lying
strictly closer to the root. If \(x\) is itself a vertex, then \(x^{(1)}\) is its parent.
Inductively, whenever \(x^{(k)}\ne0\), define
\begin{equation}
\label{eq.parent_order_def}
    x^{(k+1)}
    :=
    \bigl(x^{(k)}\bigr)^{(1)}.
\end{equation}
\begin{Lm}
\label{Lm.doubl_strip}
Let \(X\) be a \(K\)-regular tree equipped with the measure \(\mu\) defined by
\eqref{eq.m_d} and the metric \(d_\lambda\) defined by \eqref{eq.d_lambda_def}.
Assume that \(\mu\) is doubling on \(( X,d_\lambda)\). Then there exists
a constant \(C>0\), depending only on the doubling constant of \(\mu\), such that
for every \(0\le r<r'\),
\begin{equation}
\label{eq.doubling_strip}
    \mu(X^{r,r'})
    \le
    C\mu\left(\frac12 X^{r,r'}\right).
\end{equation}
\end{Lm}

\begin{proof}
Put $\tau:=\frac12\mathcal W(X^{r,r'})$.
Thus \(X^{r,r'}\) consists of the points whose radial \(d_\lambda\)-distance from the
center level \(r_c(X^{r,r'})\) is less than \(\tau\), while
\(\frac12X^{r,r'}\) consists of the points whose radial \(d_\lambda\)-distance from
the same center level is less than \(\frac{\tau}{2}\).

Let
\begin{equation}
    S_c:=\{x\in X: |x|=r_c(X^{r,r'})\}
\end{equation}
be the center level of the strip. Choose a maximal \(\frac{\tau}{2}\)-separated subset
\begin{equation}
    Z=\{z_1,\ldots,z_N\}\subset S_c
\end{equation}
with respect to the metric \(d_\lambda\). Then the balls $\left\{B_{\frac{\tau}{4}}(z_j)\right\}_{j=1}^N$
are pairwise disjoint. Moreover, since every \(z_j\) lies on the center level, we have $ B_{\frac{\tau}{4}}(z_j)\subset \frac12X^{r,r'}$ for each $j \in \{1, \ldots, N\}$.
Indeed, if \(y\in B_{\frac{\tau}{4}}(z_j)\), then the radial \(d_\lambda\)-distance from \(y\)
to the center level is at most \(d_\lambda(y,z_j)<\frac{\tau}{4}<\frac{\tau}{2}\).

By the maximality of \(Z\), the balls \(\{B_{\frac{\tau}{2}}(z_j)\}_{j=1}^N\) cover \(S_c\).
Let \(x\in X^{r,r'}\). Choose a point \(x_c\in S_c\) lying on a geodesic ray through
\(x\). Then $d_\lambda(x,x_c)<\tau$.
Since \(S_c\subset\bigcup_{j=1}^N B_{\frac{\tau}{2}}(z_j)\), there exists \(j\) such that
\(d_\lambda(x_c,z_j)<\frac{\tau}{2}\). Hence
\begin{equation}
    d_\lambda(x,z_j)
    \le
    d_\lambda(x,x_c)+d_\lambda(x_c,z_j)
    <
    \frac32\tau.
\end{equation}
Therefore
\begin{equation}
    X^{r,r'}
    \subset
    \bigcup_{j=1}^N B_{\frac32\tau}(z_j).
\end{equation}
Using the doubling property of \(\mu\), we obtain
\begin{equation}
\begin{aligned}
    \mu(X^{r,r'})
    \le
    \sum_{j=1}^N \mu\left(B_{\frac32\tau}(z_j)\right)
    \lesssim
    \sum_{j=1}^N \mu\left(B_{\frac{\tau}{4}}(z_j)\right) 
    \le
    \mu\left(\frac12X^{r,r'}\right).
\end{aligned}
\end{equation}
This proves \eqref{eq.doubling_strip}.
\end{proof}
\subsection{Fractional gradients}

We shall use the following pointwise estimate for Besov functions. In this subsection,
we work in the global setting: \((X,d,\mu)\) is a metric measure space with a globally
doubling measure \(\mu\). Thus the Besov norm defined in Section~2 is equivalent to the norm obtained by replacing the interval \((0,1)\) in \eqref{eq.bes_sem} by
\((0,\infty)\). Indeed, the contribution of large scales \(t\ge1\) is controlled by
\(\|f\|_{L_p(X)}\). Equivalently, one may use the dyadic scales \(2^{-k}\),
\(k\in\mathbb Z\).

For \(j\in\mathbb Z\) and \(f\in L_1^{\operatorname{loc}}(X)\), set
\begin{equation}
\label{eq.omega_j_def}
    \omega_j f(x)
    :=
    \inf_{c\in\mathbb R}
    \fint\limits_{B_{2^{-j}}(x)}
    |f(z)-c|\,d\mu(z).
\end{equation}

We shall use the following consequence of \cite[Lemma~2.3]{AKZ}.

\begin{Lm}
\label{Lm.AKZ_pointwise}
Let \((X,d,\mu)\) be a metric measure space with a globally doubling measure.
Then there exists a constant \(C>0\) such that, for every
\(f\in L_1^{\operatorname{loc}}(X)\), one can find a set \(N\subset X\) with
\(\mu(N)=0\) such that, for every \(k\in\mathbb Z\) and all
\(x,y\in X\setminus N\) satisfying $2^{-k-1}\le d(x,y)<2^{-k}$,
one has
\begin{equation}
\label{eq.AKZ_pointwise}
    |f(x)-f(y)|
    \le
    C\sum_{j=k-2}^{\infty}
    \bigl(\omega_j f(x)+\omega_j f(y)\bigr).
\end{equation}
\end{Lm}

\begin{Lm}
\label{Lm.special_frac_gradient}
Let \(s\in(0,1)\), \(p\in[1,\infty]\), and let \(f\in B^s_{p,1}(X)\).
For \(k\in\mathbb Z\), define
\begin{equation}
\label{eq.special_gradient_def}
    g_k(x)
    :=
    A2^{sk}
    \sum_{j=k-2}^{\infty}\omega_j f(x),
\end{equation}
where \(A>0\) is a sufficiently large structural constant. Then there exists a set
\(N\subset X\), \(\mu(N)=0\), such that, for every \(k\in\mathbb Z\) and all
\(x,y\in X\setminus N\) satisfying $2^{-k-1}\le d(x,y)<2^{-k}$,
one has
\begin{equation}
\label{eq.special_gradient_pointwise}
    |f(x)-f(y)|
    \le
    d(x,y)^s\bigl(g_k(x)+g_k(y)\bigr).
\end{equation}
Moreover, \begin{equation} \label{eq.special_gradient_norm} \sum_{k=-\infty}^{\infty} \|g_k\|_{L_p(X)} \lesssim \|f\|_{B^s_{p,1}(X)}. \end{equation}

Finally, after possibly enlarging the exceptional set \(N\), for every \(k\in\mathbb Z\)
and all \(x,y\in X\setminus N\) satisfying \(d(x,y)<2^{-k}\), we have
\begin{equation}
\label{eq.special_gradient_small_distance}
    |f(x)-f(y)|
    \le
    2^{-sk}\bigl(g_k(x)+g_k(y)\bigr).
\end{equation}
\end{Lm}

\begin{proof}
Let \(N\) be the exceptional set from Lemma~\ref{Lm.AKZ_pointwise}. Suppose first that
\(2^{-k-1}\le d(x,y)<2^{-k}\). Then \(d(x,y)^s\ge 2^{-s(k+1)}\), and therefore,
choosing \(A\) sufficiently large, \eqref{eq.AKZ_pointwise} gives
\begin{equation}
    |f(x)-f(y)|
    \le
    d(x,y)^s\bigl(g_k(x)+g_k(y)\bigr).
\end{equation}
This proves \eqref{eq.special_gradient_pointwise}.

It remains to justify \eqref{eq.special_gradient_norm}. For completeness, and because
only this case will be used below, we prove the estimate when \(q=1\). The general case
\(q\in[1,\infty]\) is contained in \cite{AKZ}. Since
\begin{equation}
    \omega_j f(x)
    \le
    \fint\limits_{B_{2^{-j}}(x)}
    |f(z)-A^\mu_{2^{-j}}f(x)|\,d\mu(z)
    =
    \overline\Delta_{2^{-j}}f(x),
\end{equation}
we have
\begin{equation}
    \|\omega_j f\|_{L_p(X)}
    \le
    \|\overline\Delta_{2^{-j}}f\|_{L_p(X)}.
\end{equation}
For \(j\ge0\), the right-hand side is controlled by the dyadic Besov seminorm.
For \(j<0\), the general oscillation estimate gives $\|\omega_j f\|_{L_p(X)}
    \lesssim
    \|f\|_{L_p(X)}$.
Hence
\begin{equation}
\label{eq.omega_sequence_est}
    \sum_{j=-\infty}^{\infty}
    2^{js}\|\omega_j f\|_{L_p(X)}
    \lesssim
    \|f\|_{B^s_{p,1}(X)}.
\end{equation}
By \eqref{eq.special_gradient_def},
\begin{equation}
    \|g_k\|_{L_p(X)}
    \le
    A2^{sk}
    \sum_{j=k-2}^{\infty}
    \|\omega_j f\|_{L_p(X)}.
\end{equation}
Consequently
\begin{equation}
\begin{aligned}
    \sum_{k=-\infty}^{\infty}
    \|g_k\|_{L_p(X)}\lesssim
    \sum_{j=-\infty}^{\infty}
    2^{js}\|\omega_j f\|_{L_p(X)} \lesssim
    \|f\|_{B^s_{p,1}(X)}.
\end{aligned}
\end{equation}
This proves \eqref{eq.special_gradient_norm}.

Finally, assume that \(d(x,y)<2^{-k}\). If \(x=y\), there is nothing to prove.
Otherwise choose \(m\ge k\) such that $2^{-m-1}\le d(x,y)<2^{-m}$.
By \eqref{eq.special_gradient_pointwise},
\begin{equation}
    |f(x)-f(y)|
    \le
    d(x,y)^s\bigl(g_m(x)+g_m(y)\bigr)
    \le
    2^{-sm}\bigl(g_m(x)+g_m(y)\bigr).
\end{equation}
On the other hand, by the definition of \(g_k\),
\begin{equation}
    2^{-sm}g_m
    =
    A\sum_{j=m-2}^{\infty}\omega_j f
    \le
    A\sum_{j=k-2}^{\infty}\omega_j f
    =
    2^{-sk}g_k.
\end{equation}
Thus
\begin{equation}
    |f(x)-f(y)|
    \le
    2^{-sk}\bigl(g_k(x)+g_k(y)\bigr),
\end{equation}
which proves \eqref{eq.special_gradient_small_distance}.
\end{proof}
\subsection{Trace theorem}

Until the end of this section, we fix \(K\ge1\) and a \(K\)-regular tree \(X\)
equipped with the measure \(\mu\) and the metric \(d_\lambda\) defined by
\eqref{eq.m_d} and \eqref{eq.d_lambda_def}. We assume that $ \mu(X)<\infty$ and $\operatorname{diam}_\lambda(X)<\infty$
and that \(\mu\) is doubling on \(( X,d_\lambda)\).

There are two natural ways to define traces of functions on \(X\) to the boundary
\(\partial X\) (see also \cite{dif_tr_op}). The first one is the metric trace:
\begin{equation}
\label{eq.tr_tr_met}
    \operatorname{Tr}_{\operatorname{met}}f(\xi)
    :=
    \lim_{r\to0}
    \fint\limits_{B_r^{\overline X}(\xi)}
    f(x)\,d\mu(x),
    \qquad \xi\in\partial X,
\end{equation}
whenever the limit exists. Since \(\mu(\partial X)=0\), the average in
\eqref{eq.tr_tr_met} may equivalently be taken over
\(B_r^{\overline X}(\xi)\cap X\).
The second one is the subtree trace:
\begin{equation}
\label{eq.tr_tr_subt}
    \operatorname{Tr}_{\Gamma}f(\xi)
    :=
    \lim_{[0,\xi)\ni x\to\xi}
    \fint\limits_{\Gamma_x}
    f(y)\,d\mu(y),
    \qquad \xi\in\partial X,
\end{equation}
whenever the limit exists.

In what follows, we study the relation between these two trace operators and describe
the trace space of the endpoint Besov space \(B^{\theta/p}_{p,1}(X)\) to the boundary
\(\partial X\).

We now introduce regularity conditions on the tree which are analogous to the upper and
lower codimension-\(\theta\) regularity conditions \eqref{eq.up_reg} and
\eqref{eq.low_reg}. For \(r\ge0\), put
\begin{equation}
\label{eq.tree_tail_functions}
    M(r)
    :=
    \mu(X^r)
    =
    \int\limits_r^\infty w(t)K^{j(t)}\,dt,
    \qquad
    \rho(r)
    :=
    \int\limits_r^\infty\lambda(t)\,dt.
\end{equation}
Here \(\rho(r)\) is the remaining \(d_\lambda\)-length from level \(r\) to the
boundary.

\begin{Def}
Let \(\theta\ge0\). We say that \((X,d_\lambda,\mu)\) is upper
\(\theta\)-regular if there exists a constant \(C>0\) such that
\begin{equation}
\label{eq.up_reg_t}
    M(r)
    \ge
    C\rho(r)^\theta,
    \qquad r\ge0.
\end{equation}
We say that \((X,d_\lambda,\mu)\) is lower \(\theta\)-regular if there exists a
constant \(C>0\) such that
\begin{equation}
\label{eq.low_reg_t}
    M(r)
    \le
    C\rho(r)^\theta,
    \qquad r\ge0.
\end{equation}
\end{Def}

\begin{Remark}
\label{Rm.reg_tr}
The preceding conditions can be reformulated in terms of subtrees. More precisely, by
the radial form of the measure \(\mu\) and by the definition of the boundary measure
\(\nu\), upper \(\theta\)-regularity is equivalent to
\begin{equation}
\label{eq.tree_upper_subtree_form}
    \nu(\partial\Gamma_x)
    \lesssim
    \frac{\mu(\Gamma_x)}
    {\operatorname{diam}_\lambda(\Gamma_x)^\theta},
    \qquad x\in X.
\end{equation}
Similarly, lower \(\theta\)-regularity is equivalent to
\begin{equation}
\label{eq.tree_lower_subtree_form}
    \nu(\partial\Gamma_x)
    \gtrsim
    \frac{\mu(\Gamma_x)}
    {\operatorname{diam}_\lambda(\Gamma_x)^\theta},
    \qquad x\in X.
\end{equation}
Indeed, for \(r=|x|\), one has $\operatorname{diam}_\lambda(\Gamma_x)\approx \rho(r)$,
and, up to constants depending only on \(K\),
\begin{equation}
    \frac{\mu(\Gamma_x)}{\nu(\partial\Gamma_x)}
    \approx
    M(r).
\end{equation}
\end{Remark}

\begin{Example}
\label{Ex.exp_tree_regular}
Let \(\lambda(t)=e^{-\varepsilon t}\) and \(w(t)=e^{-\beta t}\), where
\(\varepsilon>0\) and \(\beta>\log K\). Then
\begin{equation}
    \rho(r)\approx e^{-\varepsilon r},
    \qquad
    M(r)\approx e^{-(\beta-\log K)r}.
\end{equation}
Hence
\begin{equation}
    M(r)\approx \rho(r)^{\frac{\beta-\log K}{\varepsilon}},
\end{equation}
and therefore \((X,d_\lambda,\mu)\) is both upper and lower $\frac{\beta-\log K}{\varepsilon}$
regular.
\end{Example}
\begin{Lm}
\label{Lm.tr_equiv}
Let \(p\in[1,\infty)\), \(\theta\in(0,p)\), and put \(s=\frac{\theta}{p}\).
Assume that \((X,d_\lambda,\mu)\) is upper \(\theta\)-regular and that \(\mu\) is
doubling on \(X\). Then, for every \(f\in B^s_{p,1}(X)\), we have
\begin{equation}
\label{eq.metric_subtree_average_difference}
    \lim_{r\to0}
    \left|
    \fint\limits_{\Gamma_{x_r(\xi)}} f(y)\,d\mu(y)
    -
    \fint\limits_{B_r^{\overline X}(\xi)} f(y)\,d\mu(y)
    \right|
    =
    0
\end{equation}
for \(\nu\)-almost every \(\xi\in\partial X\). Here \(x_r(\xi)\in[0,\xi)\) is the
unique point satisfying
\begin{equation}
\label{eq.xr_def}
    \int\limits_{|x_r(\xi)|}^{\infty}\lambda(t)\,dt=r.
\end{equation}
Consequently, \(\operatorname{Tr}_{\operatorname{met}}f(\xi)\) exists if and only if
\(\operatorname{Tr}_{\Gamma}f(\xi)\) exists, for \(\nu\)-almost every
\(\xi\in\partial X\). In this case,
\begin{equation}
    \operatorname{Tr}_{\operatorname{met}}f(\xi)
    =
    \operatorname{Tr}_{\Gamma}f(\xi)
\end{equation}
for \(\nu\)-almost every \(\xi\in\partial X\).
\end{Lm}

\begin{proof}
By Lemma~\ref{Lm.completion_admissibility}, the extended measure \(\mu\) is doubling
on \((\overline X,d_\lambda)\). Take an arbitrary $f\in B^{\theta/p}_{p, 1}(X)$. Let
\(\{g_k\}_{k\in\mathbb Z}\) be the fractional \(s\)-Haj\l{}asz gradient of \(f\)
constructed in Lemma~\ref{Lm.special_frac_gradient}. Since \(f\in B^s_{p,1}(X)\), we
have
\begin{equation}
\label{eq.G_def}
    G:=\sum_{k=-\infty}^{\infty} g_k \in L_p(X).
\end{equation}
Fix \(r>0\), and choose \(k(r)\in\mathbb Z\) so that $2^{-k(r)-1}\le r<2^{-k(r)}$.
By \eqref{eq.boundary_ball_subtree_inclusions}, we have
\(
    B_r^{\overline X}(\xi)\cap X
    \subset
    \Gamma_{x_r(\xi)}.
\)
Hence, if \(y\in\Gamma_{x_r(\xi)}\) and
\(z\in B_r^{\overline X}(\xi)\cap X\), then
\(
    d_\lambda(y,z)\le 2r
\).
Therefore, the pointwise estimate from
Lemma~\ref{Lm.special_frac_gradient} gives
\begin{equation}
\begin{split}
    &\left|
    \fint\limits_{\Gamma_{x_r(\xi)}} f(y)\,d\mu(y)
    -
    \fint\limits_{B_r^{\overline X}(\xi)} f(y)\,d\mu(y)
    \right|                                                \\
    &\qquad\lesssim
    2^{-k(r)s}
    \left(
    \fint\limits_{\Gamma_{x_r(\xi)}} g_{k(r)-2}(y)\,d\mu(y)
    +
    \fint\limits_{B_r^{\overline X}(\xi)} g_{k(r)-2}(y)\,d\mu(y)
    \right).
\end{split}
\end{equation}
Using \eqref{eq.boundary_ball_subtree_inclusions}, applied with \(a=2r\), and
\eqref{eq.subtree_doubling_completion}, applied with \(a=2r\), we obtain
\begin{equation}
    \fint\limits_{\Gamma_{x_r(\xi)}} g_{k(r)-2}(y)\,d\mu(y)
    \lesssim
    \fint\limits_{B_{2r}^{\overline X}(\xi)}
    g_{k(r)-2}(y)\,d\mu(y).
\end{equation}
Therefore, by Jensen's inequality,
\begin{equation}
\label{eq.tr_coinc1}
\begin{split}
    &\left|
    \fint\limits_{\Gamma_{x_r(\xi)}} f(y)\,d\mu(y)
    -
    \fint\limits_{B_r^{\overline X}(\xi)} f(y)\,d\mu(y)
    \right|\lesssim
    \left(
    r^\theta
    \fint\limits_{B_{2r}^{\overline X}(\xi)}
    g_{k(r)-2}(y)^p\,d\mu(y)
    \right)^{1/p}                                          \\
    &\qquad\le
    \left(
    r^\theta
    \fint\limits_{B_{2r}^{\overline X}(\xi)}
    G(y)^p\,d\mu(y)
    \right)^{1/p}.
\end{split}
\end{equation}

It remains to show that the right-hand side tends to zero for \(\nu\)-almost every
\(\xi\in\partial X\). By Remark~\ref{Rm.reg_tr} and \eqref{eq.boundary_ball_subtree_inclusions}, the upper
\(\theta\)-regularity of the tree implies
\begin{equation}
\label{eq.boundary_upper_reg}
    \nu\bigl(B_{2r}^{\partial X}(\xi)\bigr)
    \lesssim
    \frac{\mu\bigl(B_{2r}^{\overline X}(\xi)\bigr)}{r^\theta},
    \qquad
    \xi\in\partial X,\quad r>0.
\end{equation}
Thus \(\nu\) is upper codimension-\(\theta\) regular on \(\partial X\) with respect to
the ambient space \((\overline X,d_\lambda,\mu)\). Hence, by
Remark~\ref{Rm.regularity}, \(\nu\ll \mathcal H_\theta\lfloor_{\partial X}\).
Since \(G^p\in L_1(X)\), Theorem~\ref{Th.leb_points} gives
\begin{equation}
    \lim_{r\to0}
    r^\theta
    \fint\limits_{B_{2r}^{\overline X}(\xi)}
    G(y)^p\,d\mu(y)
    =
    0
\end{equation}
for \(\mathcal H_\theta\)-almost every \(\xi\in\partial X\), and therefore for
\(\nu\)-almost every \(\xi\in\partial X\). Combining this with
\eqref{eq.tr_coinc1}, we obtain \eqref{eq.metric_subtree_average_difference}.

The final assertion follows immediately from
\eqref{eq.metric_subtree_average_difference}: if one of the two limits defining
\(\operatorname{Tr}_{\operatorname{met}}f(\xi)\) and
\(\operatorname{Tr}_{\Gamma}f(\xi)\) exists, then the other one exists and the two
limits are equal.
\end{proof}
\begin{Remark}
Assume that \(p\in[1,\infty)\), \(s\in(\frac{\theta}{p},1)\), and
\(q\in[1,\infty]\). Then Lemma~\ref{Lm.tr_equiv} also applies to functions in
\(B^s_{p,q}(X)\). Indeed, by the standard Besov embedding theorem, $B^s_{p,q}(X)\hookrightarrow B^{\theta/p}_{p,1}(X)$
(see \cite[Section~3]{mal_e}).
\end{Remark}

Until the end of the section, whenever both traces exist, we write
\(
    \operatorname{Tr}f
    :=
    \operatorname{Tr}_{\operatorname{met}}f
    =
    \operatorname{Tr}_{\Gamma}f\).
The trace \(\operatorname{Tr}_{\operatorname{met}}\) coincides with the trace operator
defined in the general metric-space setting, applied to the ambient space
\((\overline X,d_\lambda,\mu)\) and to the closed set \(E=\partial X\). Moreover, by Remark~\ref{Rm.reg_tr} and \eqref{eq.boundary_ball_subtree_inclusions}, upper, respectively lower,
\(\theta\)-regularity of \((X,d_\lambda,\mu)\) implies upper, respectively lower,
codimension-\(\theta\) regularity of the boundary measure \(\nu\) in the sense of
Section~2. Therefore, the boundedness and extension parts of
Theorem~\ref{Th.trees_main} follow from the general trace theorems. The remaining
tree-specific point is the following criterion for the existence of traces.

\begin{Prop}
\label{Prop.tree_upper_necessity}
Let \(p\in[1,\infty)\) and \(s\in(0,1)\). Assume that
\((X,d_\lambda,\mu)\) supports a weak \((1,p)\)-Poincar\'e inequality. Then the
following conditions are equivalent:
\begin{enumerate}
    \item for every \(f\in B^s_{p,1}(X)\), the trace \(\operatorname{Tr}f(\xi)\) is
    finite for \(\nu\)-almost every \(\xi\in\partial X\);
    \item \((X,d_\lambda,\mu)\) is upper \(ps\)-regular.
\end{enumerate}
\end{Prop}

\begin{proof}
The sufficiency follows from the direct trace theorem,
Theorem~\ref{Th.main_direct}, applied to the space
\((\overline X,d_\lambda,\mu)\), the trace set \(\partial X\), and the parameter
\(\theta=ps\). Indeed, upper \(ps\)-regularity of the tree implies upper
codimension-\(ps\) regularity of the boundary measure \(\nu\), and
Lemma~\ref{Lm.tr_equiv} identifies the metric trace with the subtree trace.

We prove the necessity. Assume that \((X,d_\lambda,\mu)\) is not upper
\(ps\)-regular. Then there exists an increasing sequence \(\{r_k\}_{k=1}^{\infty}\),
\(r_k\to\infty\), such that
\begin{equation}
\label{eq.tr_d_trees1}
    \frac{
    \int\limits_{r_k}^{\infty}w(t)K^{j(t)}\,dt
    }{
    \left(
    \int\limits_{r_k}^{\infty}\lambda(t)\,dt
    \right)^{ps}
    }
    <
    2^{-k}.
\end{equation}
Passing to a subsequence, we may also assume that
\begin{equation}
\label{eq.tr_d_trees2}
    \int\limits_{r_{k+1}}^{\infty}\lambda(t)\,dt
    \le
    \frac12
    \int_{r_k}^{\infty}\lambda(t)\,dt, \qquad \delta_k:=\int\limits_{r_k}^{\infty}\lambda(t)\,dt<1, \qquad k\in \mathbb{N}.
\end{equation}

Put
\begin{equation}
    l_k
    :=
    \int\limits_{r_k}^{r_{k+1}}\lambda(t)\,dt.
\end{equation}
Let \(r_k'<r_k''\) be the unique numbers in \((r_k,r_{k+1})\) satisfying
\begin{equation}
    \int\limits_{r_k}^{r_k'}\lambda(t)\,dt
    =
    \int\limits_{r_k''}^{r_{k+1}}\lambda(t)\,dt
    =
    \frac{l_k}{4}.
\end{equation}
Define \(\psi_k:[0,\infty)\to[0,1]\) by
\begin{equation}
\label{eq.psi_k_def}
    \psi_k(t)
    =
    \begin{cases}
    \displaystyle
    \frac{4}{l_k}\int\limits_{r_k}^{t}\lambda(\tau)\,d\tau,
    & t\in[r_k,r_k'],\\[2ex]
    1,
    & t\in[r_k',r_k''],\\[1ex]
    \displaystyle
    \frac{4}{l_k}\int\limits_{t}^{r_{k+1}}\lambda(\tau)\,d\tau,
    & t\in[r_k'',r_{k+1}],\\[2ex]
    0,
    & \text{otherwise}.
    \end{cases}
\end{equation}
Then the radial function $h_k(x):=\psi_k(|x|)$
is Lipschitz on \(X\), supported in \(X^{r_k,r_{k+1}}\), and satisfies
\begin{equation}
\label{eq.tr_d_trees4}
    \operatorname{lip} h_k(x)
    \lesssim
    \frac{1}{\delta_k}
    \chi_{X^{r_k,r_{k+1}}}(x).
\end{equation}
Here we used \eqref{eq.tr_d_trees2}, which implies $l_k
    \approx \delta_k$.

By Remark~\ref{Rm.lip_est}, applied with $\delta_k$,
and by \eqref{eq.tr_d_trees1}, we obtain
\begin{equation}
\label{eq.tr_d_trees5}
\begin{split}
    \|h_k\|_{B^s_{p,1}(X)}
    &\lesssim
    \delta_k^{1-s}\|\operatorname{lip} h_k\|_{L_p(X)}
    +
    \delta_k^{-s}\|h_k\|_{L_p(X)}
    +
    \|h_k\|_{L_p(X)}                                      \\
    &\lesssim
    \delta_k^{-s}\left(\mu(X^{r_k})\right)^{1/p}                  =
    \left(
    \frac{
    \int\limits_{r_k}^{\infty}w(t)K^{j(t)}\,dt
    }{
    \left(
    \int\limits_{r_k}^{\infty}\lambda(t)\,dt
    \right)^{ps}
    }
    \right)^{1/p}
    <
    2^{-\frac{k}{p}}.
\end{split}
\end{equation}

Let \(c_k=k\), and define
\begin{equation}
\label{eq.bad_tree_function_def}
    f(x)
    :=
    \sum_{k=1}^{\infty}
    c_k h_k(x)
    =
    \sum_{k=1}^{\infty}
    c_k\psi_k(|x|).
\end{equation}
Since
\begin{equation}
    \sum_{k=1}^{\infty} k\,2^{-\frac{k}{p}}<\infty,
\end{equation}
estimate \eqref{eq.tr_d_trees5} gives $f\in B^s_{p,1}(X)$.

It remains to show that the trace of \(f\) fails to be finite at every boundary point.
Fix \(\xi\in\partial X\), and let \(x_k\in[0,\xi)\) be such that \(|x_k|=r_k\).
By radiality of the measure \(\mu\), for every \(j\ge k\) we have
\begin{equation}
    \frac{
    \mu\bigl(\Gamma_{x_k}\cap \frac12 X^{r_j,r_{j+1}}\bigr)
    }{
    \mu(\Gamma_{x_k})
    }
    =
    \frac{
    \mu\bigl(\frac12 X^{r_j,r_{j+1}}\bigr)
    }{
    \mu(X^{r_k})
    }.
\end{equation}
Hence, using Lemma~\ref{Lm.doubl_strip}, we obtain
\begin{equation}
\label{eq.bad_trace_divergence}
\begin{split}
    \fint\limits_{\Gamma_{x_k}}f(y)\,d\mu(y)
    &\ge
    \frac{1}{\mu(X^{r_k})}
    \sum_{j=k}^{\infty}
    c_j\mu\left(\frac12 X^{r_j,r_{j+1}}\right)       \\
    &\gtrsim
    \frac{1}{\mu(X^{r_k})}
    \sum_{j=k}^{\infty}
    c_j\mu\left(X^{r_j,r_{j+1}}\right)               \\
    &\ge
    c_k
    \frac{
    \sum_{j=k}^{\infty}\mu\left(X^{r_j,r_{j+1}}\right)
    }{
    \mu(X^{r_k})
    }
    =
    c_k
    \to\infty
\end{split}
\end{equation}
as \(k\to\infty\). Therefore, \(\operatorname{Tr}f(\xi)\) is not finite. Since
\(\xi\in\partial X\) was arbitrary, the trace fails to be finite at every boundary point.
This proves the necessity.
\end{proof}
\begin{proof}[Proof of Theorem~\ref{Th.trees_main}]
The boundedness of the trace operator under upper \(\theta\)-regularity follows from
Theorem~\ref{Th.main_direct}, applied to the completed ambient space
\((\overline X,d_\lambda,\mu)\), the trace set \(E=\partial X\), and the boundary
measure \(\nu\). Indeed, by Lemma~\ref{Lm.completion_admissibility}, the completed
space has the required doubling property, and upper \(\theta\)-regularity of the tree
implies upper codimension-\(\theta\) regularity of \(\nu\). Lemma~\ref{Lm.tr_equiv}
identifies the metric trace with the subtree trace.

The equivalence between the existence of finite traces for all functions in
\(B^s_{p,1}(X)\) and upper \(ps\)-regularity is exactly
Proposition~\ref{Prop.tree_upper_necessity}.

Finally, if the tree is both upper and lower \(\theta\)-regular and supports a weak
\((1,p)\)-Poincar\'e inequality, then Lemma~\ref{Lm.completion_admissibility} implies
that \((\overline X,d_\lambda,\mu)\) is \(p\)-admissible. Moreover, the lower
\(\theta\)-regularity of the tree implies lower codimension-\(\theta\) regularity of
\(\nu\). Hence the extension operator follows from Theorem~\ref{Th.main_inverse}.
\end{proof}

\begin{Remark}
For Newton--Sobolev spaces on \(K\)-regular trees, one often uses the pointwise trace
operator
\begin{equation}
    \operatorname{Tr}f(\xi)
    =
    \lim_{[0,\xi)\ni x\to\xi} f(x),
\end{equation}
that is, the limit of \(f\) along the geodesic ray \([0,\xi)\).

For Besov spaces the situation is more delicate. In general, Besov functions need not
have locally continuous representatives, and therefore there is no canonical pointwise
representative to evaluate along a ray. For instance, if
\(w(t)=e^{-\beta t}\) and \(\lambda(t)=e^{-\varepsilon t}\), then away from the vertices the tree is locally one-dimensional, and local continuity follows only under the usual one-dimensional
Besov embedding condition (see, for example, \cite[Section~2.8.3]{trieb})
\begin{equation}
    s>\frac1p
    \qquad\text{or}\qquad
    s=\frac1p,\ q=1.
\end{equation}
This condition is unrelated to the
codimension of the boundary, which is the relevant parameter in the trace problem
considered here. For this reason, we use averaged trace operators rather than
pointwise traces in the Besov setting.
\end{Remark}

\subsection*{Acknowledgments}
\addtocontents{toc}{\protect\setcounter{tocdepth}{0}}
The author is deeply grateful to his Scientific Advisor Alexander I. Tyulenev for proposing the problem, and for his guidance and attention throughout this work.
\par
\renewcommand{\bibname}{References}
\bibliographystyle{plain}
\bibliography{refs.bib}

@article{AKZ, author = {Gogatishvili, Amiran and Koskela, Pekka and Zhou, Yuan}, title = {Characterizations of {B}esov and {T}riebel--{L}izorkin spaces on metric measure spaces}, journal = {Forum Mathematicum}, volume = {25}, number = {4}, pages = {787--819}, year = {2013}, doi = {10.1515/form.2011.135}, eprint = {1106.2561}, archivePrefix = {arXiv}, primaryClass = {math.CA}, url = {https://arxiv.org/abs/1106.2561} }

@article{Gul,
  author  = {Gulisashvili, A. B.},
  title   = {Traces of differential functions on subsets of the {Euclidean} space},
  journal = {J. Sov. Math.},
  fjournal = {Journal of Soviet Mathematics},
  issn    = {0090-4104},
  volume  = {42},
  number  = {2},
  pages   = {1573--1583},
  year    = {1988},
  doi     = {10.1007/BF01665043},
  language= {English}
}

@misc{mal_e,
  author       = {Mal{\'y}, Luk{\'a}{\v{s}}},
  title        = {Trace and extension theorems for {Sobolev}-type functions in metric spaces},
  year         = {2017},
  eprint       = {1704.06344},
  archivePrefix = {arXiv},
  note   = {\href{https://arxiv.org/abs/1704.06344}{arXiv:1704.06344}},
  primaryClass = {math.MG}
}

@article{bur_gold,
  author   = {Burenkov, V. I. and Gol'dman, M. L.},
  title    = {On the extension of functions from {$L_p$}},
  journal  = {Tr. Mat. Inst. Steklova},
  fjournal = {Trudy Matematicheskogo Instituta Imeni V. A. Steklova},
  issn     = {0371-9685},
  volume   = {150},
  pages    = {31--51},
  year     = {1979},
  language = {Russian}
}

@article{gagl,
  author    = {Gagliardo, Emilio},
  title     = {Caratterizzazioni delle tracce sulla frontiera relative ad alcune classi di funzioni in {$n$} variabili},
  journal   = {Rend. Semin. Mat. Univ. Padova},
  fjournal  = {Rendiconti del Seminario Matematico della Universit\`a di Padova},
  year      = {1957},
  volume    = {27},
  pages     = {284--305},
  url       = {https://eudml.org/doc/106977},
  language  = {Italian}
}

@article{mig,
  author    = {Marcos, Miguel Andr\'es},
  title     = {A trace theorem for {B}esov functions in spaces of homogeneous type},
  journal   = {Publicacions Matem\`atiques},
  volume    = {62},
  number    = {1},
  pages     = {185--211},
  year      = {2018},
  publisher = {Universitat Aut\`onoma de Barcelona, Departament de Matem\`atiques},
  doi       = {10.5565/PUBLMAT6211810},
  url       = {https://doi.org/10.5565/PUBLMAT6211810}
}

@article{saks,
  author  = {Saksman, Eero and Soto, Tom\'as},
  title   = {Traces of {B}esov, {T}riebel--{L}izorkin and {S}obolev {S}paces on {M}etric {S}paces},
  journal = {Analysis and Geometry in Metric Spaces},
  volume  = {5},
  number  = {1},
  pages   = {98--115},
  year    = {2017},
  doi     = {10.1515/agms-2017-0006},
  url     = {https://doi.org/10.1515/agms-2017-0006}
}

@article{Ihnat,
  author   = {Ihnatsyeva, Lizaveta and V{\"a}h{\"a}kangas, Antti V.},
  title    = {Characterization of traces of smooth functions on {A}hlfors regular sets},
  journal  = {Journal of Functional Analysis},
  volume   = {265},
  number   = {9},
  pages    = {1870--1915},
  year     = {2013},
  issn     = {0022-1236},
  doi      = {10.1016/j.jfa.2013.07.006},
  url      = {https://www.sciencedirect.com/science/article/pii/S0022123613002656}
}

@misc{Jon,
  author       = {Jonsson, Alf and Wallin, Hans},
  title        = {Function spaces on subsets of $\mathbb{R}^n$},
  year         = {1984},
  howpublished = {Math. Rep., Chur 2, No. 1, xiv + 221 p.}
}

@book{sob_mms,
  author    = {Heinonen, Juha and Koskela, Pekka and Shanmugalingam, Nageswari and Tyson, Jeremy T.},
  title     = {Sobolev Spaces on Metric Measure Spaces: An Approach Based on Upper Gradients},
  series    = {New Mathematical Monographs},
  publisher = {Cambridge University Press},
  place     = {Cambridge},
  year      = {2015},
  collection= {New Mathematical Monographs}
}

@article{tyul,
  author  = {Tyulenev, Alexander I.},
  title   = {Traces of {Sobolev} spaces to irregular subsets of metric measure spaces},
  journal = {Sbornik: Mathematics},
  volume  = {214},
  number  = {9},
  pages   = {1241--1320},
  year    = {2023},
  doi     = {10.4213/sm9893e},
  eprint  = {2212.11271},
  archivePrefix = {arXiv},
  primaryClass  = {math.FA}
}

@article{shanm_tr_s,
  author    = {Bj\"orn, Anders and Bj\"orn, Jana and Gill, James T. and Shanmugalingam, Nageswari},
  title     = {Geometric analysis on {C}antor sets and trees},
  journal   = {Journal f\"ur die reine und angewandte Mathematik (Crelles Journal)},
  number    = {725},
  pages     = {63--114},
  year      = {2014},
  doi       = {10.1515/crelle-2014-0099},
  url       = {http://dx.doi.org/10.1515/crelle-2014-0099},
  ISSN      = {0075-4102},
  publisher = {Walter de Gruyter GmbH}
}

@article{tr_den_tr,
  author   = {Koskela, Pekka and Nguyen, Khanh Ngoc and Wang, Zhuang},
  title    = {Trace and density results on regular trees},
  journal  = {Potential Anal.},
  fjournal = {Potential Analysis},
  issn     = {0926-2601},
  volume   = {57},
  number   = {1},
  pages    = {101--128},
  year     = {2022},
  doi      = {10.1007/s11118-021-09907-2}
}

@article{dyad_norm,
  author   = {Koskela, Pekka and Wang, Zhuang},
  title    = {Dyadic norm {Besov}-type spaces as trace spaces on regular trees},
  journal  = {Potential Anal.},
  fjournal = {Potential Analysis},
  issn     = {0926-2601},
  volume   = {53},
  number   = {4},
  pages    = {1317--1346},
  year     = {2020},
  doi      = {10.1007/s11118-019-09808-5}
}

@article{Ext_tr_h_f,
  title     = {Extension and trace results for doubling metric measure spaces and their hyperbolic fillings},
  volume    = {159},
  ISSN      = {0021-7824},
  url       = {http://dx.doi.org/10.1016/j.matpur.2021.12.003},
  DOI       = {10.1016/j.matpur.2021.12.003},
  journal   = {Journal de Math\'ematiques Pures et Appliqu\'ees},
  publisher = {Elsevier BV},
  author    = {Bj\"orn, Anders and Bj\"orn, Jana and Shanmugalingam, Nageswari},
  year      = {2022},
  pages     = {196--249}
}

@article{dif_tr_op,
  author  = {Koskela, Pekka and Nguyen, Khanh Ngoc and Wang, Zhuang},
  title   = {Trace {O}perators on {R}egular {T}rees},
  journal = {Analysis and Geometry in Metric Spaces},
  volume  = {8},
  number  = {1},
  pages   = {396--409},
  year    = {2020},
  doi     = {10.1515/agms-2020-0117},
  url     = {https://doi.org/10.1515/agms-2020-0117}
}

@article{koch_sn,
  author    = {Kazaniecki, Krystian and Wojciechowski, Micha{\l}},
  title     = {Trace Operator on von {K}och's {S}nowflake},
  journal   = {Potential Analysis},
  year      = {2024},
  volume    = {61},
  pages     = {659--684},
  doi       = {10.1007/s11118-024-10124-w},
  url       = {https://doi.org/10.1007/s11118-024-10124-w},
  publisher = {Springer},
  issn      = {0926-2601}
}

@article{tyul2,
  author  = {Tyulenev, Alexander I.},
  title   = {Traces of {Sobolev} spaces on piecewise {A}hlfors--{D}avid regular sets},
  journal = {Mathematical Notes},
  volume  = {114},
  number  = {3},
  pages   = {351--376},
  year    = {2023},
  doi     = {10.1134/S0001434623090079},
  eprint  = {2303.08913},
  archivePrefix = {arXiv},
  primaryClass  = {math.FA}
}

@book{trieb,
 author = {Triebel, Hans},
 title = {Theory of function spaces},
 fseries = {Mathematik und ihre Anwendungen in Physik und Technik},
 series = {Math. Anwend. Phys. Tech.},
 volume = {38},
 year = {1983},
 publisher = {Akademische Verlagsgesellschaft Geest \& Portig K.-G., Leipzig},
 language = {English},
 zbMATH = {3870002},
 Zbl = {0546.46028}
}

@article{Peetre,
 author = {Peetre, Jaak},
 title = {A counter-example connected with {Gagliardo}'s trace theorem},
 fjournal = {Annales Societatis Mathematicae Polonae},
 journal = {Commentationes Mathematicae. Special Issue},
 issn = {0373-8299},
 pages = {277--282},
 year = {1979},
 language = {English},
 zbMATH = {3690060},
 Zbl = {0442.46026}
}

@article{Bes_or,
 author = {Besov, O. V.},
 title = {Investigation of a family of function spaces in connection with theorems of imbedding and extension},
 fjournal = {Translations. Series 2. American Mathematical Society},
 journal = {Transl., Ser. 2, Am. Math. Soc.},
 issn = {0065-9290},
 volume = {40},
 pages = {85--126},
 year = {1964},
 language = {English},
 doi = {10.1090/trans2/040/03},
 zbMATH = {3254379},
 Zbl = {0158.13901}
}

@article{Iwon,
 author = {Piotrowska, Iwona},
 title = {Traces on fractals of function spaces with {Muckenhoupt} weights},
 fjournal = {Functiones et Approximatio. Commentarii Mathematici},
 journal = {Funct. Approximatio, Comment. Math.},
 issn = {0208-6573},
 volume = {36},
 pages = {95--117},
 year = {2006},
 language = {English},
 doi = {10.7169/facm/1229616444},
 zbMATH = {5225674},
 Zbl = {1140.46015}
}

@article{Har_Sch,
author = {Dorothee D. Haroske and Hans-J{\"u}rgen Schmeisser},
title = {On trace spaces of function spaces with a radial weight: the atomic approach},
journal = {Complex Variables and Elliptic Equations},
volume = {55},
number = {8-10},
pages = {875--896},
year = {2010},
publisher = {Taylor \& Francis},
doi = {10.1080/17476930903276050},
URL = { https://doi.org/10.1080/17476930903276050},
eprint = {https://doi.org/10.1080/17476930903276050}
}

@article{gibara,
      title={Solving a Dirichlet problem for unbounded domains via a conformal transformation}, 
      author={Ryan Gibara and Riikka Korte and Nageswari Shanmugalingam},
      journal = {Mathematische Annalen},
      volume  = {389},
      year    = {2024},
      pages   = {2857--2901},
      eprint={2209.09773},
      archivePrefix={arXiv},
      primaryClass={math.AP},
      url={https://arxiv.org/abs/2209.09773}, 
}

@article{Gol,
  author  = {Gol'dman, M. L.},
  title   = {On the extension of functions from {$L_p(\mathbb R^m)$} to a space of higher dimension},
  journal = {Mathematical Notes},
  year    = {1979},
  volume  = {25},
  number  = {4},
  pages   = {266--270},
  doi     = {10.1007/BF01688477},
  note    = {Russian original: Mat. Zametki 25 (1979), no. 4, 513--520}
}

@article{Netrusov,
  author   = {Netrusov, Yu. V.},
  title    = {Sets of singularities of functions in spaces of {Besov} and {Lizorkin--Triebel} type},
  journal  = {Trudy Mat. Inst. Steklov.},
  volume   = {187},
  year     = {1989},
  pages    = {162--177},
  note     = {English translation: Proc. Steklov Inst. Math. 187 (1990), 185--203},
  mrnumber = {1006450},
  zbl      = {0719.46018}
}

@article{Nuutinen,
  author  = {Nuutinen, Juho},
  title   = {The {Besov} capacity in metric spaces},
  journal = {Annales Polonici Mathematici},
  volume  = {117},
  number  = {1},
  pages   = {59--78},
  year    = {2016},
  issn    = {1730-6272},
  doi     = {10.4064/ap3843-4-2016},
  url     = {https://doi.org/10.4064/ap3843-4-2016}
}

\end{document}